\documentclass{article}

\usepackage{microtype}
\usepackage{graphicx}
\usepackage{subcaption}
\usepackage{booktabs} 

\usepackage{hyperref}

\usepackage[accepted]{icml2025}

\usepackage{amsmath}
\usepackage{amssymb}
\usepackage{mathtools}
\usepackage{amsthm}
\usepackage{bm}
\usepackage{bbm}
\usepackage{constants}

\usepackage[capitalize,noabbrev]{cleveref}

\theoremstyle{plain}
\newtheorem{theorem}{Theorem}[section]

\newtheorem{lemma}[theorem]{Lemma}

\theoremstyle{definition}

\newtheorem{assumption}[theorem]{Assumption}
\theoremstyle{remark}

\newtheorem{example}[theorem]{Example}

\DeclareMathOperator{\statdim}{stat.dim}
\DeclareMathOperator{\E}{\mathbb{E}}
\DeclareMathOperator{\PP}{\mathbb{P}}
\DeclareMathOperator{\iid}{\overset{\mathrm{iid}}{\sim}}

\DeclareMathOperator{\prox}{prox}

\DeclareMathOperator{\dist}{dist}

\DeclareMathOperator{\argmin}{argmin}

\newcommand{\env}{\mathsf{env}}

\newcommand{\indep}{\perp \!\!\! \perp}
\newcommand{\R}{\mathbb{R}}
\newcommand{\bv}{\bm{v}}
\newcommand{\bu}{\bm{u}}
\newcommand{\bg}{\bm{g}}
\newcommand{\bb}{\bm{b}}

\newcommand{\bx}{\bm{x}}
\newcommand{\by}{\bm{y}}
\newcommand{\bp}{\bm{p}}

\newcommand{\bw}{\bm{w}}

\newcommand{\ma}{\mathcal{A}}
\newcommand{\ml}{\mathcal{L}}
\newcommand{\mg}{\mathcal{G}}
\newcommand{\mh}{\mathcal{H}}

\def\a{a}

\usepackage[textsize=tiny]{todonotes}

\icmltitlerunning{Phase transitions for the existence of unregularized M-estimators
in single index models}

\begin{document}

\twocolumn[
\icmltitle{Phase transitions for the existence of unregularized M-estimators \\
in single index models}



\icmlsetsymbol{equal}{*}

\begin{icmlauthorlist}
\icmlauthor{Takuya Koriyama}{chicago}
\icmlauthor{Pierre C. Bellec}{rutgers}

\end{icmlauthorlist}

\icmlaffiliation{chicago}{University of Chicago, Booth School of Business}
\icmlaffiliation{rutgers}{Department of Statistics, Rutgers University}

\icmlcorrespondingauthor{Takuya Koriyama}{tkoriyam@uchicago.edu}

\icmlkeywords{
proportional asymptotics, CGMT, conic geometry, convex analysis in Hilbert space
}

\vskip 0.3in
]



\printAffiliationsAndNotice{}  

\begin{abstract}
This paper studies phase transitions for the existence of unregularized M-estimators under proportional asymptotics where the sample size $n$ and feature dimension $p$ grow proportionally with $n/p \to \delta \in (1, \infty)$. We study the existence of M-estimators in single-index models where the response $y_i$ depends on covariates \(\bm{x}_i \sim N(0, I_p)\) through an unknown index \(\bm{w} \in \mathbb{R}^p\) and an unknown link function. An explicit expression is derived for the critical threshold $\delta_\infty$ that determines the phase transition for the existence of the M-estimator, generalizing the results of \citet{candes2020phase} for binary logistic regression to other single-index models.
Furthermore, we investigate the existence of a solution to the nonlinear system of equations governing the asymptotic behavior of the M-estimator when it exists. The existence of solution to this system for $\delta > \delta_\infty$ remains largely unproven outside the global null in binary logistic regression. We address this gap with a proof that the system admits a solution if and only if $\delta > \delta_\infty$, providing a comprehensive theoretical foundation for proportional asymptotic results that require as a prerequisite the existence of a solution to the system.
\end{abstract}

\section{Introduction}
Let $(\bx_i,y_i)_{i\in[n]}$ be a sample of i.i.d. observations
where $\bx_i\in\R^p$ and follows the normal distribution $\bx_i\sim N(0,I_p)$. 
We consider responses $y_i\in\mathcal Y$
where $\mathcal Y\subset \R$ 
and assume a single-index model of the form
$$
\PP(y_i \le t \mid \bx_i)= F(t, \bx_i^\top \bw),
\qquad
\forall t\in \R$$
where $F:\R\times \R\to [0,1]$ is an unknown deterministic function 
 and $\bw\in\R^p$ is an unknown index with $\|\bw\|=1$.
This includes generalized linear models, such as
the Poisson model (in which $\mathcal Y=\{0,1,2,\dots\}$)
or Binary logistic model (in which $\mathcal Y=\{-1,1\}$ for instance).

An unregularized M-estimator is fit to observed data
$(y_i,\bx_i)_{i\in[n]}$ by the minimization problem
\begin{equation}
\label{infimum}
\inf_{\bb\in \R^p} \sum_{i=1}^n \ell_{y_i}(\bx_i^\top\bb)
\end{equation}
where $\ell_{y}(\cdot)$ is convex for every $y\in\mathcal Y$.
If the infimum is achieved, we say the M-estimator exists
and denote it by $\hat{\bb}$, i.e., 
$$
\hat{\bb}\in \argmin_{\bb\in \R^p} \sum_{i=1}^n \ell_{y_i}(\bx_i^\top \bb).
$$
In this paper we focus on a high dimensional regime where the sample size and feature grow proportionally as
$$
\frac{n}{p} \to \delta\in (1, +\infty)
$$
Here the constant $\delta$ quantifies sample size per dimensions.  In this proportional asymptotic regime, under logistic regression model with binary
response $\mathcal Y=\{-1,1\}$
and with $\ell_{y}(t)=\log(1+\exp(-yt))$
the logistic loss, the seminal work of \citet{candes2020phase}
establishes that the existence of the M-estimator undergoes
a sharp phase transition at a critical threshold $\delta_{\infty}$:
\begin{itemize}
\item
    if $\delta > \delta_{\infty}$ then the M-estimator exists
with high-probability (i.e., the infimum in \eqref{infimum} is attained),
while
\item
    if $\delta < \delta_{\infty}$ then the M-estimator does not exist
(i.e., the infimum is not attained) with high-probability.
\end{itemize}
\begin{figure*}[ht]
    \centering
    \begin{subfigure}{0.3\textwidth}
        \centering
        \includegraphics[width=\textwidth]{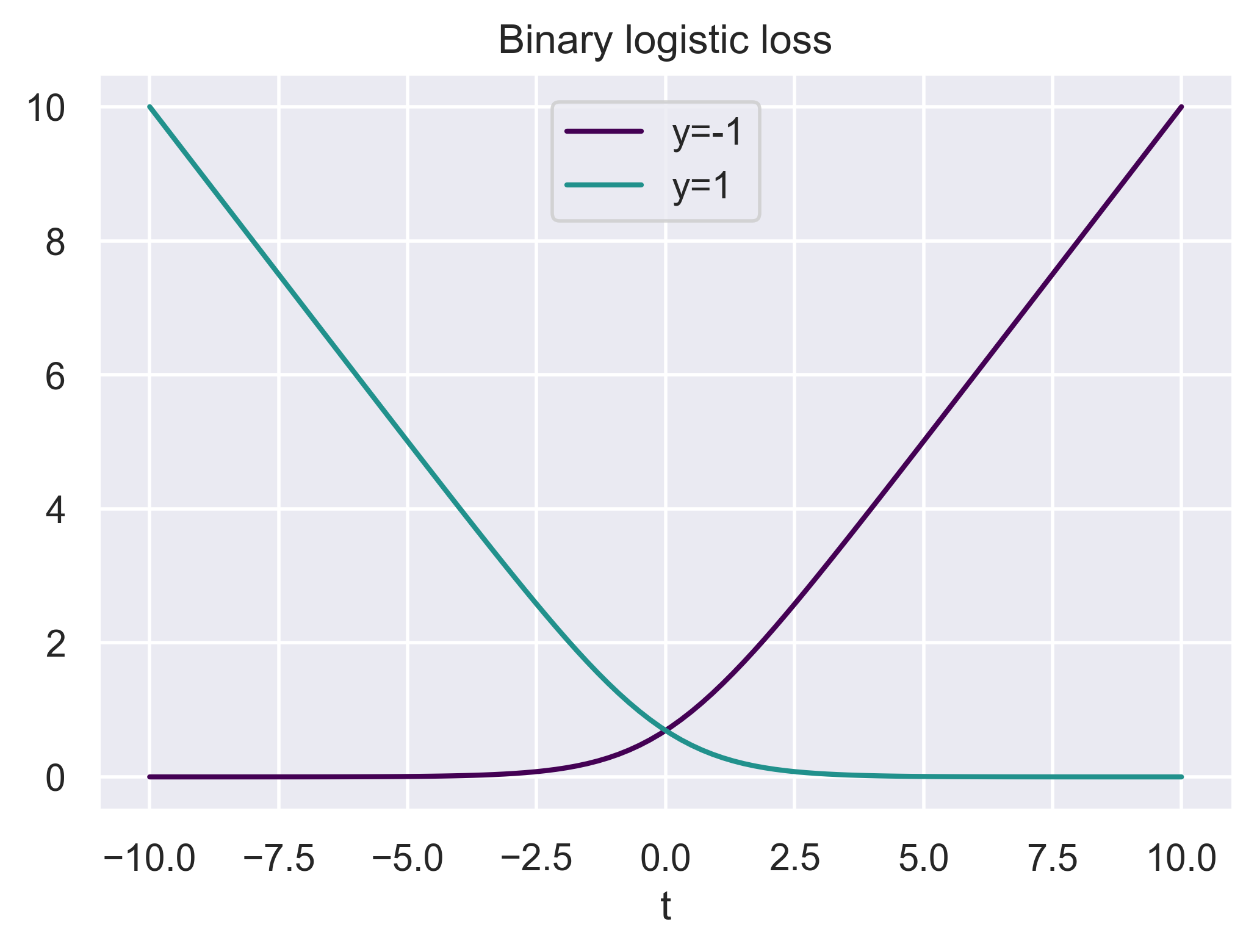} 
        \caption{Logistic loss}
    \end{subfigure}
    \begin{subfigure}{0.3\textwidth}
        \centering
        \includegraphics[width=\textwidth]{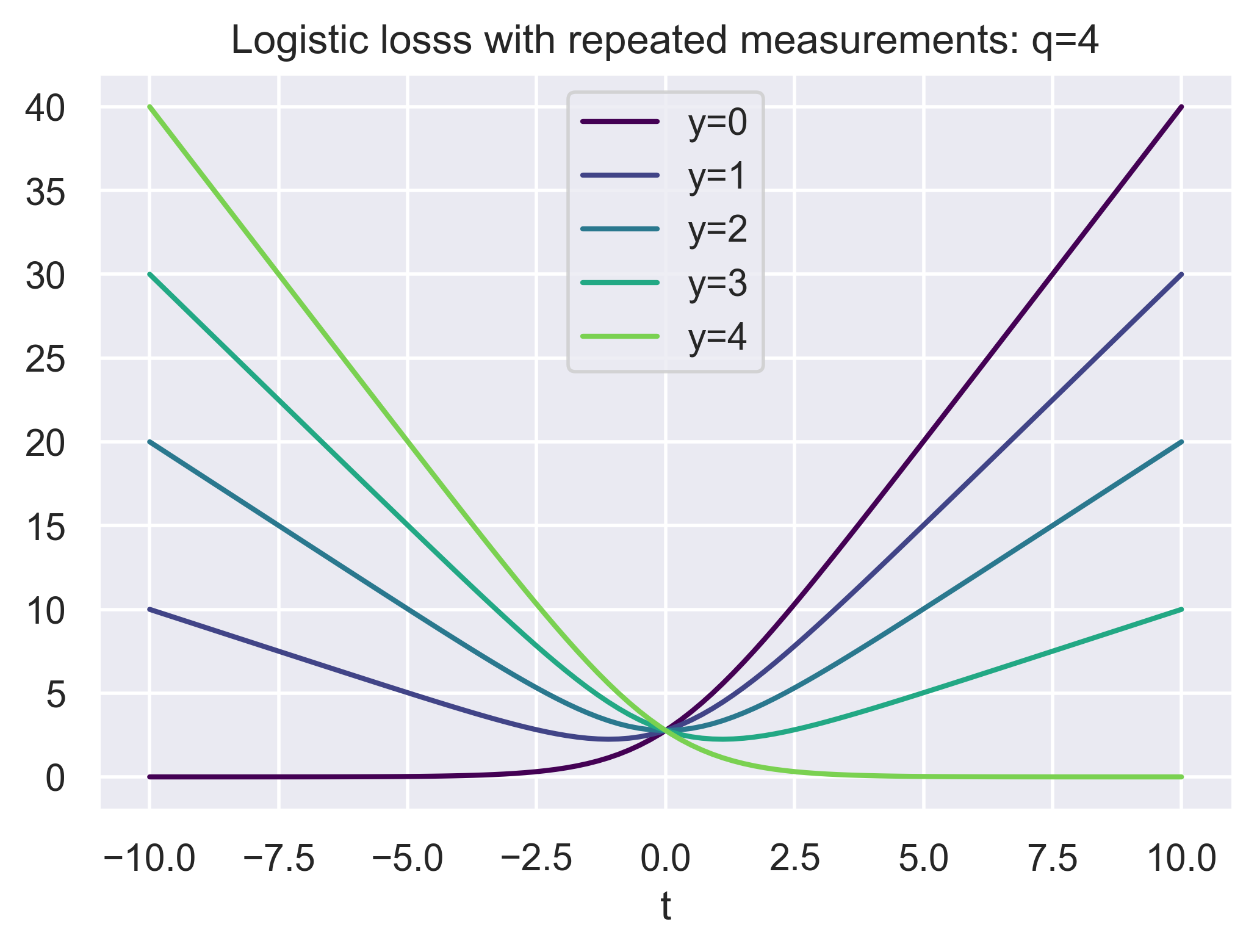} 
        \caption{Binomial loss}
    \end{subfigure}
    \begin{subfigure}{0.3\textwidth}
        \centering
        \includegraphics[width=\textwidth]{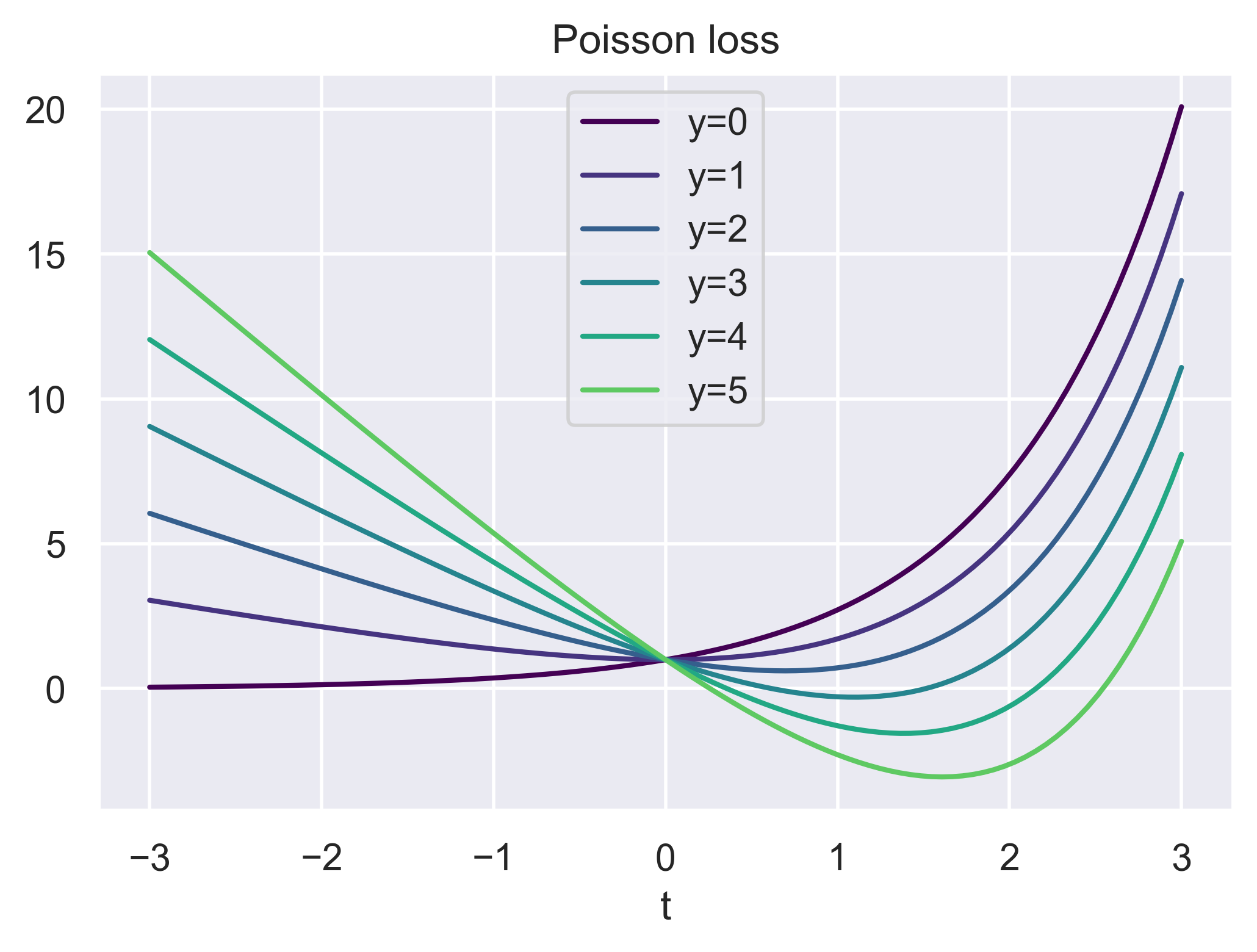} 
        \caption{Poisson}
    \end{subfigure}
    \caption{Three examples of loss functions.}
    \label{fig:horizontal}
\end{figure*}

If $\delta > \delta_{\infty}$,
the behavior of the unregularized M-estimator $\hat{\bb}$,
including the limit in probability of 
$\hat{\bb}^\top\bw$ and the limit in probability of
$\|(I_p-\bw\bw^\top)\hat{\bb}\|_2$, is characterized
\cite{sur2018modern}
by the nonlinear system with three unknowns $(\gamma,a,\sigma)\in \R_{>0}\times \R\times \R_{>0}$:
\begin{align}
    \begin{split}\label{eq:system}
        \gamma^{-2} \delta^{-1} \sigma^2 &=  \E[\ell_Y'(\prox[\gamma\ell_Y](\a U+\sigma G))^2],\\
        0 &= \E[U \ell_Y'(\prox[\gamma \ell_Y](\a U+\sigma G))],\\
        \sigma(1-\delta^{-1}) &= \E[G \prox[\gamma \ell_Y](\a U + \sigma G)], 
    \end{split}
\end{align}
where $G\sim N(0,1)$ is independent
of $(U,Y)$, and $(U,Y)$ has the same distribution 
as $(\bx_i^\top \bw, y_i)$. In particular $U\sim N(0,1)$.
The phase transition result of \citet{candes2020phase}
for Gaussian design and binary logistic regression
has been extended by \citet{tang2020existence}
to elliptic covariate distributions and general binary
response models.
\citet{han2022gaussian} extended the phase transition of the
logistic regression model with Gaussian covariates to
the constrained minimization setting where
the infimum in \eqref{infimum} is restricted to a closed convex cone.
{
More broadly, the phase transition phenomena studied in our paper are connected to earlier works in statistical physics (\cite{cover1965geometrical, gardner1988optimal, krauth1989storage}). Especially, \cite{cover1965geometrical} analyzed the geometry of linear inequalities and derived $\delta_\infty=2$ under the null, i.e., when $\bx_i$ is independent of $y_i$. 
More recently, similar phase transition behavior has been investigated in \cite{mignacco2020role} for Gaussian mixtures models and in \cite{gerace2020generalisation} for random feature models. 
}

Results such as \citet{sur2018modern,salehi2019impact}
studying the behavior of the M-estimator on the side
of the phase transition where it exists with high-probability
assume that the nonlinear system \eqref{eq:system}
admits a unique solution.
Under the global null in binary logistic regression, \citet{sur2019likelihood} establishes that
$\delta_\infty = 2$ and that the system
\eqref{eq:system} admits a unique solution if and only if
$\delta > 2$.
Beyond the global null, it was observed \cite{sur2018modern}
that the system \eqref{eq:system} can be solved numerically
if $\delta > \delta_{\infty}$ (where $\delta_{\infty}$ is characterized
in \citet{candes2020phase}),
and that if the solution exists, it is unique (see Remark 2 of the supplement of \citet{sur2018modern}).
However to our knowledge there is no
proof yet that the system admits a solution for $\delta > \delta_{\infty}$
except under the global null
(see discussion after eq. (16) of the supplement of \citet{sur2018modern}).

The goal of the present paper is twofold:
\begin{itemize}
    \item
        To characterize the critical threshold
        $\delta_\infty$ for Gaussian covariates beyond
        binary response models, for instance the Poisson model.
    \item
        To prove that the system \eqref{eq:system} admits a unique solution
        if and only if $\delta>\delta_\infty$.
\end{itemize}

\section{Main Result}

Let us introduce the three examples of interest
that our assumptions will cover.

\begin{example}[Binary logistic regression]
    \label{example:binary}
    Here, labels in $\mathcal Y=\{-1,1\}$ and the loss function 
    $$
    \ell_y(t) = \log(1+\exp(-yt)).
    $$
\end{example}

\begin{example}[Logistic regression with repeated measurements]
    \label{example:multinomial}
    Let $q\ge 2$.
    Here, $\mathcal Y=\{0,1,\dots,q\}$ and the loss function,
    corresponding to a binomial regression model with $q$
    throws and a sigmoid link function for the probability, is
    $$
    \ell_y(t) = q\log(1+\exp(t)) - y t.
    $$
    If $q=1$, this is equivalent to binary logistic regression
    by renaming $\{1,0\}$ to $\{1,-1\}$.
\end{example}

\begin{example}[Poisson regression]
    \label{example:poisson}
    If the labels are in $\mathcal Y=\{0,1,2,3,\dots\} =\mathbb N$,
    the non-negative likelihood of Poisson generalized linear model
    leads to the loss function
    $$
    \ell_y(t) = \exp(t) - y t.
    $$
\end{example}

The phenomenon of the phase transition for the existence
of the M-estimator comes from the lack of coercivity 
of some of the loss functions
$\ell_{y_i}$ appearing in the optimization problem
\eqref{infimum}.
For instance, 
\begin{itemize}
\item
In the binary logistic regression case
(\Cref{example:binary}), the loss
$\ell_{y_i}$ is not coercive for all $y_i$:
it is increasing for $y_i=-1$ and decreasing for $y_i=1$.
\item
For binomial logistic regression with $q\ge 2$ measurements,
(\Cref{example:multinomial}), the loss
$\ell_{y_i}$ is coercive if $y_i\in\{1,...,q-1\}$,
increasing for $y_i=0$ and decreasing for $y_i=q$.
\item
    For Poisson regression, (\Cref{example:poisson}),
the loss $\ell_{y_i}$ is coercive for $y_i\ge 1$
and increasing for $y_i=0$.
\end{itemize}

The values of $y_i$ leading to a coercive, increasing or
decreasing loss $\ell_{y_i}(\cdot)$ 
and the distribution of $(\bx_i^\top \bw, y_i)$ will determine
the critical threshold $\delta_\infty$.
In order to study $\delta_\infty$, it will be thus useful
to introduce the following notation:
For a random variable $Y$ valued in $\mathcal Y$, we define
the events
$\Omega_\vee(Y), \Omega_\nearrow(Y), \Omega_\searrow(Y)$
by
\begin{align}\label{eq:def_Omega}
    \begin{split}
        \Omega_\vee(Y) &=  \{\text{$\ell_Y(\cdot)$ is coercive}\},\\
        \Omega_\nearrow(Y) &= \{\text{$\ell_Y$ is strictly increasing}\},\\
        \Omega_\searrow(Y) &= \{\text{$\ell_Y$ is strictly decreasing}\}.
    \end{split}
\end{align}

These events can be made explicit for the three examples thanks
to the discussion in the three bullet points above:

\begin{itemize}
\item
In the binary logistic regression case
(\Cref{example:binary}) we have
$\Omega_\nearrow(Y)=\{Y=-1\}$ and $\Omega_\searrow(Y)=\{Y=1\}$.
\item
In \Cref{example:multinomial},
$\Omega_\nearrow(Y)=\{Y=0\}$ as well as $\Omega_\searrow(Y)=\{Y=q\}$ 
and $\Omega_\vee(Y)=\{0<Y<q\}$.

\item
    For Poisson regression (\Cref{example:poisson}) we have
    $\Omega_\nearrow(Y)=\{Y=0\}$ and $\Omega_\vee(Y)=\{Y\ge 1\}$.
\end{itemize}

We will assume that $\ell_Y$ is strictly convex in our working
assumptions, so that if $\ell_Y$ is not
coercive it must be either strictly increasing or strictly decreasing.
We first state an assumption that prevents the problem from being trivial.

\begin{assumption}
    \label{assum:so_not_trivial}
    The loss $\ell_y$ is strictly convex for every $y\in\mathcal Y$,
    and the law of $Y$ satisfies
    \begin{align*}
        \begin{split}
            &\E\Bigl[U^2 \bigl(I\{\Omega_\vee\} + I\{\Omega_\searrow, U>0\}+ I\{\Omega_\nearrow, U<0\}\bigr)\Bigr] > 0, \\
          & \E\Bigl[U^2 \bigl(I\{\Omega_\vee\} +  I\{\Omega_\searrow, U<0\} + I\{\Omega_\nearrow, U>0\}\bigr)\Bigr] > 0     
        \end{split}
    \end{align*}
    where we omit the argument $Y$ in the three events
    $\Omega_\vee(Y), \Omega_\nearrow(Y)$ and $\Omega_\searrow(Y)$
    for brevity.
\end{assumption}

This prevents the problem from being trivial in the following sense.
Consider $p=1$, and write $x_i=U_i$ (which is now scalar valued).
Assume that the first line in \Cref{assum:so_not_trivial}
is 0. 
The minimization problem \eqref{eq:infinite_dimensional_optimization}
becomes
$$
\inf_{a\in\R} \sum_{i=1}^n \ell_{y_i}(a U_i).
$$
Since $\PP(U^2>0)=1$ by $U\sim N(0,1)$, this implies $\PP(\Omega_\vee(Y))=0$,
so for each of the $n$ terms, the loss $\ell_{y_i}(\cdot)$
is not coercive. 
In fact, each term is increasing in $a$, because
$U_i > 0 \Rightarrow \Omega_\nearrow(y_i)$
and
$U_i < 0 \Rightarrow \Omega_\searrow(y_i)$
by the first line in \Cref{assum:so_not_trivial},
so in this case $a\mapsto\sum_{i=1}^n \ell_{y_i}(a U_i)$ is increasing
with probability one and the infimum is never attained.
Similarly, if the second line in \Cref{assum:so_not_trivial}
is zero, then 
$a\mapsto\sum_{i=1}^n \ell_{y_i}(a U_i)$ is decreasing and the
infimum is never attained.

In conclusion, \Cref{assum:so_not_trivial} merely assumes that
for $p=1$, the infimum is achieved, i.e., the M-estimator exists, {with positive probability}.
If the M-estimator does not exist for $p=1$ then it will not exist
for $p>1$ either, so \Cref{assum:so_not_trivial} is required
to avoid this trivial case that the M-estimator does not exist for all $p\ge
1$.
The remaining of our working assumptions are given below.

\begin{assumption}\label{assumption}
    The loss $\ell_Y$ satisfies the following:
\begin{enumerate}
    \item For all $y\in \mathcal{Y}$, $\ell_y:\R\to \R$ is $C^1$, strictly convex, and not constant. 
    \item $\E[|\inf_u\ell_Y(u)|] <+\infty$ and  $\E[|\ell_Y(G)|]<+\infty$ where $G\sim N(0,1)$ independent of $Y$. 
    \item $\PP(\Omega_\vee)<1$. 
    \item There exists a positive constant $b$ and $\sigma(Y)$-measurable positive random variable $D(Y)$ satisfying $\E[D(Y)]<+\infty$ and $\E[D^2(Y)\Omega_{\vee}(Y)]<+\infty$ 
     such that 
    \begin{align*}
        \forall u\in\R, \quad 
        \ell_Y(u) \ge -D(Y) + \frac{1}{b} \times  
        \begin{cases}
            u  &  \text{under }\Omega_{\nearrow}\\
            |u|& \text{under }\Omega_{\vee}\\ 
            -u &  \text{under }\Omega_{\searrow}
    \end{cases}
    \end{align*}
\end{enumerate}
\end{assumption}


Beyond \Cref{assum:so_not_trivial},
\Cref{assumption}
requires differentiability of the loss (item 1),
mild integrability conditions (item 2),
that $\ell_{y_i}$ is not always coercive (item 3)
and 
that in the directions where $\ell_Y$ diverges to $+\infty$,
it does so at least as fast as an affine function with slope
$1/b$ and squared integrable intercept $D(Y)$ (item 4).

We define the critical threshold $\delta_\infty\in (0, +\infty]$ by
\begin{align}\label{eq:threshold_definition}
    \frac{1}{\delta_\infty} \coloneq  \inf_{t\in\R} \varphi(t),
\end{align}
where $\varphi:\R\to\R_{\ge 0}$ is the convex function defined as  
\begin{align*}
    \varphi(t) \coloneq & \phantom{+} 
    \E\Bigl[\Bigl(G+Ut\Bigr)^2 I\{\Omega_\vee\}\Bigr] \\
    &
+ \E\Bigl[\Bigl(G+Ut\Bigr)_{+}^2 I\{\Omega_\nearrow\}\Bigr]\\
&
+ \E\Bigl[\Bigl(G+Ut\Bigr)_{-}^2 I\{\Omega_\searrow\}\Bigr]. 
\end{align*}
Here, the positive part of any real $a$ is denoted $a_+=\max(0,a)$
and the negative part $a_-=\max(-a, 0)$ {
and the square is always taken after the positive/negative parts, i.e., $a_+^2
= (a_+)^2$ and $a_-^2 = (a_-)^2$.
}
If $\inf_{t\in\R} \varphi(t)=0$
then $\delta_\infty$ is interpreted as $\delta_\infty=+\infty$. 
Above, the infimum over $t\in\R$ always admits a minimizer $t_*\in\R$ under \Cref{assum:so_not_trivial} (see \Cref{lemma:threshold_exist}).

Our first result is that the threshold $\delta_\infty$ characterizes the phase transition regarding whether the M-estimator exists or not with high-probability under the proportional regime $n/p\to\delta$. 
\begin{theorem}\label{theorem:phase_transition_conic_geometry}
    As $n, p\to +\infty$ with $n/p\to \delta$, we have 
\begin{align*}
    \PP\left(\parbox{10em}{\ The M-estimator exists, \\
    \ i.e., the inf \eqref{infimum} is attained}\right) 
    \to \begin{cases}
        1 & \text{if }\delta > \delta_\infty, \\
        0 & \text{if }\delta < \delta_{\infty}.
    \end{cases}
\end{align*}
\end{theorem}

\begin{figure}
    \centering
    \includegraphics[width=0.4\textwidth]{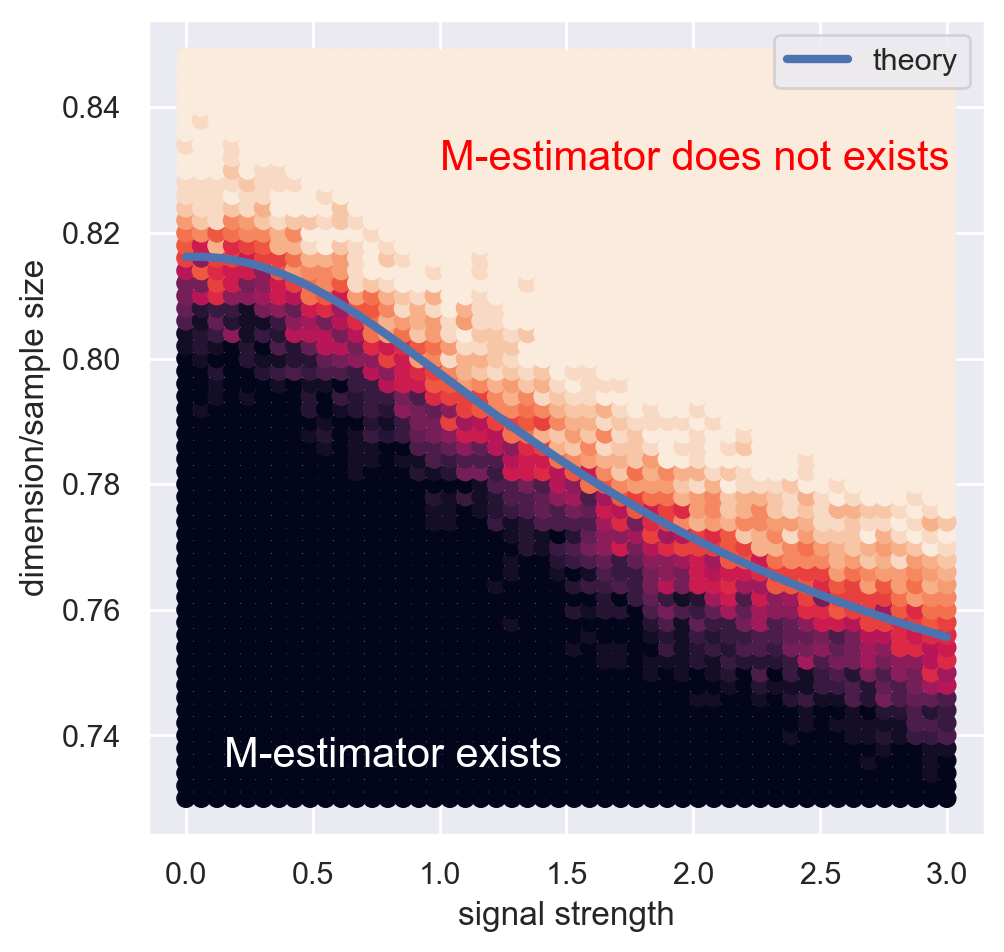}
    \caption{Count of instances where the minimizer in \eqref{infimum} exists for varying $p/n$ and signal strength. 
        Simulation parameter: $n=1500$, $20$ repetitions, $\ell_y(u)=e^u-yu$
    is the Poisson loss, $y_i\mid \bx_i$ satisfies the Poisson
model \eqref{poisson_model}.}
    \label{fig:poisson_phase_transition}
\end{figure}
\Cref{theorem:phase_transition_conic_geometry} is numerically verified by \Cref{fig:poisson_phase_transition} for Poisson regression (see \Cref{sec:simulation} for details).
The proof of \Cref{theorem:phase_transition_conic_geometry} is based on a common argument based on conic geometry and the Gaussian kinematic formula given by \citet{amelunxen2014living}. This use of the kinematic formula
of \citet{amelunxen2014living} is similar to the 
argument in \citet{candes2020phase} for the binary logistic regression model.

A more technical question, that requires an investigation
beyond the application of the kinematic formula of \citet{amelunxen2014living},
is whether the critical threshold also characterizes
the phase transition
regarding the existence of solution to the nonlinear system 
\eqref{eq:threshold_definition}. 
A formal proof of 
existence of a unique solution to \eqref{eq:threshold_definition}
is important, as it is required to leverage the Convex Gaussian
Minmax Theorem (CGMT) of \citet{thrampoulidis2018precise}.
For instance, the works \citet{salehi2019impact,loureiro2021learning}
which apply the CGMT in generalized linear models to study
$\hat{\bb}$, assume in their theorems that the system \eqref{eq:system}
admits a unique solution.

\begin{theorem}\label{theorem:phase_transition_system}
\hspace{4mm}
\begin{itemize}
    \item If $\delta\le \delta_{\infty}$, the system \eqref{eq:system}
    has no solution.
    \item If $\delta > \delta_{\infty}$, the system \eqref{eq:system}
    has a unique solution.
\end{itemize}
\end{theorem}
To our knowledge, a proof of this relationship between
the critical threshold $\delta_\infty$ and the existence of a
solution to the system \eqref{eq:system} is new, even in the 
case of binary logistic regression. The global null case was handled
in \cite{sur2019likelihood};
however, outside the global null case,
this phenomenon was observed numerically in \cite{sur2018modern} without
proof---see the discussion after eq. (16) of the supplement of
\cite{sur2018modern}.

{
\citet{montanari2019generalization} generalized the threshold
$\delta_\infty$ of \citet{candes2020phase}
for linear separation in binary classification (or equivalently
for the existence of \eqref{eq:equivalent_condition_mle}
with the logistic loss), allowing 
an arbitrary single-index model for $\PP(y_i=1\mid \bx_i)$.
\citet{montanari2019generalization} further established
the existence of a unique solution
to a system governing the behavior of max-margin classifiers
for any $\delta<\delta_\infty$, that is, for any $\delta$ such that
the data is linearly separable with high-probability.
We emphasize that the system studied in \citet{montanari2019generalization}
is different and complementary from  the system \eqref{eq:system} of interest
in the present paper: The system 
of \citet{montanari2019generalization} governs the behavior of max-margin
classifiers for $\delta<\delta_\infty$ while the system \eqref{eq:system} studied here governs
the behavior of the M-estimator \eqref{infimum} for which \Cref{theorem:phase_transition_system} establishes existence of solutions if and only if
$\delta > \delta_\infty$.
}

The tools we use to obtain \Cref{theorem:phase_transition_system}
are based on the existence of a solution
to a convex minimization problem in an infinite-dimensional Hilbert space
which is the focus of the next section.

\section{Infinite Dimensional Optimization Problem}
In this section, we define the mathematical objects
at the heart of the proof of \Cref{theorem:phase_transition_system}
and outline the proof strategy.
The key is the analysis of an infinite-dimensional convex optimization problem that is in a dual relationship with the nonlinear system. {
The use of such infinite-dimensional optimization problems to
prove the existence of solutions to nonlinear systems of equations
was pioneered in \citet{montanari2019generalization} and later
used in the context of boosting and L1 interpolation \citep{liang2022precise},
to analyse the Lasso \citep{celentano2020lasso},
and for robust regularized regression \cite{bellec2023existence}.
}

The notation and setup of this infinite-dimensional convex optimization
problem is heavily inspired by \citet{bellec2023existence},
which studies the existence of solutions to systems of a similar
nature to \eqref{eq:system} in robust regression.
In the robust regression setup of \citet{bellec2023existence}
where coercive and Lipschitz loss functions $\ell_{y_i}(\cdot)$ are considered,
the corresponding system has always a unique solution
and the M-estimator always exists: there is no phase transition.
A novelty of the present paper is to explain how these tools,
in particular the infinite-dimensional optimization below,
can be used to predict the phase transition for the existence
of a minimizer in \eqref{infimum} and the phase transition for the
existence of solutions to the system \eqref{eq:system}.

To begin with, let us consider the almost sure equivalent classes $\mh$
of squared integrable measurable functions of $(G, U, Y)$
$$
\mh=\{v:\R^3\to\R: \E[v(G, U, Y)^2]<+\infty\}, 
$$
where $G\sim N(0,1)$ and independent of $(U, Y)$. Here $(U, Y)$ is equal in distribution to $(y_i, \bx_i^\top\bw)$. 
Almost sure equivalence classes of $\mh$ form a Hilbert space 
equipped with the usual inner product $\langle u,v\rangle \coloneq \E[u(G,U, Y)v(G,U, Y)]$ and corresponding Hilbert norm
$\|v\| = \sqrt{\langle v, v\rangle}$.
We will sometimes refer to $\mh$ itself as the Hilbert space,
in this case we implicitly identify random variables
$v(G, U, Y)$ that are equal almost surely.
{
For brevity, inside an expectation and probability signs
with respect to the probability measure of $(G, U, Y)$,
we simply write $v$ to denote the random variable $v(G,U,Y)$.
For instance, we write simply $\E[vG]$ instead of $\E[v(G, U, Y)G]$.
}

Now we define two functions $\mg$ and $\ml$ as follows:
\begin{align*}
    \mg: \mh\to \R, \quad v\mapsto \|v\| - \E[vG]/\sqrt{1-\delta^{-1}}
\end{align*}
and 
\begin{align*}
    &\ml: \R\times \mh \to \R\cup \{+\infty\},\\
    &(a, v) \mapsto \begin{cases}
        \E[\ell_Y(aU+v)] & \text{if } \E[|\ell_Y(aU+v)|]<+\infty\\
        +\infty & \text{otherwise}
    \end{cases}
\end{align*}
Here, $\ml$ is a proper lower semicontinuous convex function, while $\mg$ is a Lipschitz, finite valued, and convex function (See \Cref{lm:objective_func_property} and \Cref{lm:constraint_property}). 
With these functions $(\ml, \mg)$, we claim that the system of nonlinear equations \eqref{eq:system} admits a unique solution if and only if the following infinite-dimensional convex optimization problem over $\R\times \mh$
\begin{align}\label{eq:infinite_dimensional_optimization}
    \min_{(a, v)\in\R\times \mh} \ml(a, v) \quad \text{subject to} \quad \mg(v) \le 0    
\end{align}
admits a unique minimizer $(a_*, v_*)\in \R\times \mh$. We will make this point more precise in the next paragraph. 

The key to such an equivalence between the nonlinear system \eqref{eq:system} and infinite-dimensional optimization problem \eqref{eq:infinite_dimensional_optimization} is the existence of the Lagrange multiplier associated with the constraint $\mg(v)\le 0$. 
By Proposition 27.31 of \citet{bauschke2017convex},
an element $(a_*, v_*)\in\R\times \mh$ solves the constrained optimization problem $\min_{(a,v): \mg(v)\le 0}\ml(v)$ if and only if there exists a Lagrange multiplier $\mu_*\ge 0$ such that the KKT condition 
\begin{align}\label{eq:KKT_main}
    \begin{split}
        -\mu_* \partial \mg(v_*) \cap \partial_v \ml(a_*, v_*) &\ne \emptyset\\
        \partial_a \ml(a, v)&\ni 0 \\
        \mu_*\mg(v_*)&=0\\
        \mg(v_*) &\le 0
    \end{split}
\end{align}
is satisfied, where $\partial \mg$ and $\partial\ml$ are the subdifferentials of the convex functions $\mg, \ml$. Furthermore, we will argue that the Lagrange multiplier $\mu_*$ is strictly positive. Combined with $\mu_*\mg(v_*)=0$ in the KKT condition \eqref{eq:KKT_main}, this means that $\mg(v_*)=0$, i.e., the constraint $\mg(v)\le 0$ is binding. 
{
Following \citet{bellec2023existence}, equipped with
}
this positive Lagrange multiplier $\mu_*>0$ and the binding condition $\mg(v_*)=0$, we establish the following equivalence between the minimizer of the optimization problem \eqref{eq:infinite_dimensional_optimization} and the solution to the nonlinear system of equations \eqref{eq:system}.

\begin{theorem}[Equivalence]\label{theorem:equivalence_system_optimization}
\hspace{4mm}
\begin{itemize}
    \item Suppose $(a_*, v_*)\in \R\times \mh$ solves the constrained optimization problem \eqref{eq:infinite_dimensional_optimization} with $\|v_*\|>0$. Let $\mu_*$ be the Lagrange multiplier satisfying the KKT condition. Let $(\gamma_*, \sigma_*)$ be the positive scalar defined by
    $$
    \sigma_* = \|v_*\|/\sqrt{1-\delta^{-1}} > 0, \quad \gamma_* = \|v_*\|/\mu_*>0.
    $$
    Then the pair $(a_*, \sigma_*, \gamma_*)$ solves the nonlinear system of equation. 

    \item  If $(a_*, \gamma_*, \sigma_*)\in \R \times \R_{>0}\times \R_{>0}$ satisfies the nonlinear system, letting 
    $$
    v_* = \prox[\gamma_* \ell_Y](a_*U + \sigma_*G)-a_*U 
    $$
    $(a_*, v_*)$ solves the optimization problem \eqref{eq:infinite_dimensional_optimization} with $\|v_*\| = \sigma_* \sqrt{1-\delta^{-1}}>0$ and the KKT condition \eqref{eq:KKT_main} is satisfied for $\mu_*=\sigma_*\sqrt{1-\delta^{-1}}/\gamma_*>0$. 
\end{itemize}
\end{theorem}
\Cref{theorem:equivalence_system_optimization} implies that the nonlinear system of equations \eqref{eq:system} admits a unique solution $(a_*, \gamma_*, \sigma_*)\in \R \times \R_{>0}\times \R_{>0}$ if and only if the optimization problem $\min_{\mg(v) \le0}\ml(a,v)$ admits a unique solution $(a_*, v_*)\in\R\times \mh$ with $v_*\ne 0$ and a unique Lagrange multiplier $\mu_*>0$ satisfying the KKT condition \eqref{eq:KKT_main}. {In order to apply \Cref{theorem:equivalence_system_optimization},
we need to establish that the degenerate case $v_*=0$ cannot happen.}
\begin{lemma}(Non-degeneracy)\label{lm:nonzero}
If $(a_*, v_*)$ solves the optimization problem \eqref{eq:infinite_dimensional_optimization}, then $v_*\ne 0$. 
\end{lemma}
    In the proof of \Cref{lm:nonzero}, the differentiability of the loss $\ell_y$ is crucial in preventing the degenerate case $v_* = 0$. When the loss is not differentiable, another different threshold, $\delta_{\text{perfect}}$, emerges to determine whether $v_* = 0$ or $v_* \ne 0$ occurs
\citep{bellec2023existence}.   

\begin{lemma}(Uniqueness)\label{lm:unique}
    The minimizer of the optimization problem \eqref{eq:infinite_dimensional_optimization} is unique if it exists. Furthermore, the Lagrange multiplier satisfying the KKT condition \eqref{eq:KKT_main} is also unique. 
\end{lemma}
Combining \Cref{theorem:equivalence_system_optimization}, \Cref{lm:nonzero}, and \Cref{lm:unique}, we conclude that the system of nonlinear equations has a unique solution if and only if the infinite-dimensional optimization problem admits a minimizer. {
This equivalence is useful because studying the
existence of solutions to the system \eqref{eq:system} 
directly is a tenuous analysis problem that has been solved in only
a few cases: for the Lasso \cite{bayati2011lasso,miolane2021distribution},
or for the global null case of logistic regression \cite{sur2019likelihood}.
Instead of studying the system \eqref{eq:system} directly,
this equivalence allows us to focus on the existence of minimizer
for the infinite-dimensional convex minimization problem
\eqref{eq:infinite_dimensional_optimization}.
Even though the problem is infinite-dimensional, there is a well-developed
theory for convex minimization in Hilbert spaces
\cite{bauschke2017convex} which can be leveraged to study
\eqref{eq:infinite_dimensional_optimization}; including the KKT condition
or the fact that a coercive convex objective function admits a minimizer.

It remains to establish that the existence of a minimizer for the optimization problem is governed by the threshold $\delta_\infty$ (which is the same as in \Cref{theorem:phase_transition_conic_geometry}), thus completing the proof of \Cref{theorem:phase_transition_system}. 
}
\begin{theorem}\label{theorem:phase_transition}
    Let $\delta_\infty$ be the threshold defined in \eqref{eq:threshold_definition}. Then we have the following:
   \begin{itemize}
       \item If $\delta \le  \delta_{\infty}$, the problem
           \eqref{eq:infinite_dimensional_optimization} has no minimizer. 
       \item If $\delta > \delta_{\infty}$, the problem \eqref{eq:infinite_dimensional_optimization} admits a minimizer. 
   \end{itemize}
\end{theorem}
{
Let us explain where this phase transition comes from,
starting from the case $\delta \le \delta_\infty$ where we claim
that \eqref{eq:infinite_dimensional_optimization} admits no
minimizer.
The first idea concerns $\mathcal L$: a natural avenue to
show that there is no minimizer is to try to find a direction
$(a,v)$ such that $t\mapsto \ell_Y(s (aU + v))$ is decreasing
in $s>0$ for all realizations of $(Y, U)$
(we are looking for such a direction because if a convex function
admits a ray along which it is decreasing, then it admits no minimizer).
By considering the three events $\Omega_\vee(Y), \Omega_\nearrow(Y), \Omega_\searrow(Y)$, this motivates the definition of 
the cone $\mathcal{C} \subset \R \times \mh$ defined as 
\begin{align}
    (a, p)\in \mathcal{C} \Leftrightarrow aU + p 
    \begin{cases}
        \le 0  &\text{under }\Omega_{\nearrow}(Y)\\
        = 0  & \text{under }\Omega_{\vee} (Y)\\
        \ge 0  &\text{under }\Omega_{\searrow} (Y).
    \end{cases}
\end{align}
Next, the direction we are looking for should also
satisfy the constraint $\mg(v)\le 0$. This motivates the consideration
of
$$
(a_*,p_*) \in \argmin_{(a,p)\in\mathcal C} \E[(G - p)^2]
$$
because among all $(a,p)\in\mathcal C$ such that $\|p\|=\|p^*\|$,
the $p^*$ defined above necessarily has larger correlation $\E[Gp^*]$ 
with $G$, so $p^*$ has a better chance to satisfy the constraint
$\mg(p^*)\le 0$ than $p$. We show in \Cref{lemma:threshold_exist} that
$$
\|p_*\| = \sqrt{1-\delta_\infty^{-1}},
\qquad
\|p_*\|^2 = \E[Gp^*] .
$$
This immediately gives that if $\delta\le \delta_\infty$ then
$p^*$ satisfies the constraint $\mg(p_*)\le0$ in \eqref{eq:infinite_dimensional_optimization} (this is serendipitous and ``barely'' works out,
since for any $\delta > \delta_\infty$ the constraint would be violated).
If $\delta \le \delta_\infty$, we have exhibited a direction $(a_*, p_*)$ such that
$\mathcal L(s a_*, s p_*)$ is decreasing in $s>0$
and such that $\mg( s p_*) = s \mg(p_*)\le 0$.
Since we have found a direction along which the objective function
is decreasing and which satisfies the constraint, the minimization
problem admits no minimizer. All these arguments are made precise
and formally proved in the appendix.
}

For the other side of the phase transition, $\delta > \delta_\infty$,
the following idea is used, which exhibits a similarly serendipitous
phenomenon. Here we must show that \eqref{eq:infinite_dimensional_optimization}
admits a minimizer. This is typically obtained by showing
that the objective function is coercive (i.e., has bounded level sets):
that any $v,a$ satisfying
the constraint such that $\ml(a,v)\le \xi$ for some $\xi\in\R$
must satisfy $|a| + \|v\|\le C(\xi)$ for some constant $C(\xi)$.
We break the problem by breaking $v$ into two parts,
$v = \tilde v + (v-\tilde v)$
where
\begin{align*}
   \tilde{v} =-aU + \begin{cases}
       (aU+v)_{-} & \text{under }\Omega_{\nearrow}(Y)\\
        (aU+v)_{+} & \text{under }\Omega_{\searrow}(Y)\\
        0 & \text{under }\Omega_{\vee}(Y)
    \end{cases}
\end{align*}
so that $(a, \tilde{v})\in \mathcal{C}$ for all $(a, v)\in\R\times \mh$. 
Here, $aU+ \tilde v$ is the additive part of $aU+v$ that satisfies the constraints
in the definition of $\mathcal C$ and carries a risk of generating
a ray along which the objective function is decreasing (as for $p_*$ in
the previous paragraph). On the other hand, the other additive part
$v-\tilde v$ can be bounded using the one-sided 
coercivity of $\ell_Y$ in $\Omega_\nearrow(Y)$ or $\Omega_\searrow(Y)$,
and the two-sided coercivity under $\Omega_\vee(Y)$
(this is made precise and formally proved in the appendix, see
\Cref{lm:minimizer_exist}).
To bound $\tilde v$ (or directly $v$) after having controlled $v-\tilde v$,
we establish using the properties of $p_*$ the inequality
\begin{align*}
\|v\|
(\sqrt{1-\delta^{-1}} - \|p_*\|)
&\le \E[v(G-p_*)]
\\&\le\E[(\tilde v - v)(p_* - G)],
\end{align*}
see \eqref{25_serependipitous} in the appendix.
The factor
$(\sqrt{1-\delta^{-1}} - \|p_*\|)$ in the left-hand side
is positive only on the side $\delta>\delta_\infty$ of the phase transition,
which serendipitously lets us prove that $\|v\|$ is in turn bounded,
that the objective
function is \eqref{eq:infinite_dimensional_optimization}
is coercive, that  \eqref{eq:infinite_dimensional_optimization}
consequently admits a minimizer, and by the equivalence in 
\Cref{theorem:equivalence_system_optimization} that the 
system \eqref{eq:infinite_dimensional_optimization}
admits a solution. This strategy is made precise and formally proved
in \Cref{appendix_E_proof_3.4}.

\begin{figure*}
    \centering
    \includegraphics[width=0.96\textwidth]{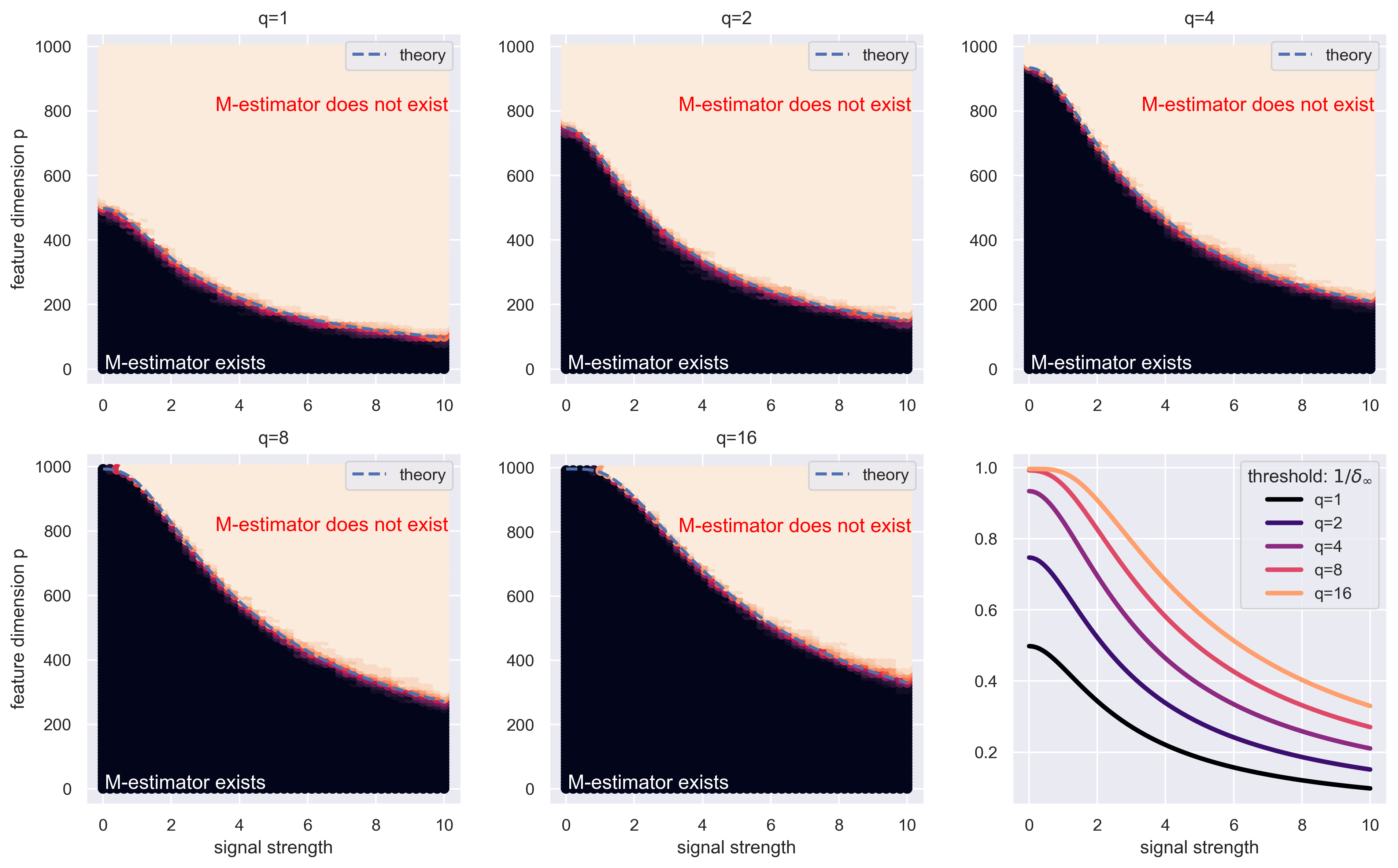}
    \caption{Count of instances where the minimizer in \eqref{infimum} exists for varying $p/n$ and signal strength $\kappa$. 
        Simulation parameter: $n=1000$, $20$ repetitions, $y_i\mid \bx_i \sim$ satisfies the binomial model $\operatorname{Binomial}(q, p_i)$ as in \eqref{eq:q_logistic_model}. 
        }
    \label{fig:q_logistic_phase_transition}
\end{figure*}

\section{Numerical Simulation}\label{sec:simulation}
We generate the covariates $(\bx_i)_{i=1}^n \iid N(\bm{0}_p, I_p)$ and responses $y_i\mid \bx_i$ according to the Poisson model
\begin{equation}
    \label{poisson_model}
\forall k\in \mathbb{N}, \quad 
\PP\Bigl(y_i = k \mid \bx_i\Bigr) = \frac{\lambda_i^k}{k!} \exp(-\lambda_i), 
\end{equation}
where $\lambda_i = \exp(-\kappa \bm{e}_1^\top \bx_i)$. 
Here, $\bm{e}_1$ is the first canonical basis vector, and $\kappa \ge 0$ is the signal strength. We fix $n=1000$ and for varying values of $(p/n, \kappa)$, we generate $20$ datasets of $(\bx_i, y_i)_{i=1}^n$. For each dataset $(\bx_i, y_i)_{i=1}^n$, we solve the optimization problem $\inf_{\bb\in\R^p} \sum_{i=1}^n \ell_{y_i}(\bx_i^\top\bb)$ using the Poisson loss $\ell_{y_i}(t)=\exp(t)-y_i t$ and record whether a minimizer exists using linear programming.
In \Cref{fig:poisson_phase_transition}, we normalize the count of instances where a minimizer exists by dividing by $20$, with the black points indicating higher rates of existence. Additionally, we plot the theoretical threshold $1/\delta_\infty$ defined in \eqref{eq:threshold_definition} and compare it with the empirical result. The theoretical threshold effectively separates the two regions, delineating where the M-estimator exists and where it does not.

Next, generate $y_i|\bx_i$ according to the Binomial distribution
\begin{align}\label{eq:q_logistic_model}
    \forall k\in [q], \quad \PP(y_i=k|\bx_i) = {q \choose k} p_i^{k} (1-p_i)^{q-k},  
\end{align}
where $p_i = \frac{1}{1+\exp(-\kappa \bm{e}_1^\top \bx_i)}$. Here, $q\in\{1, 2, \dots \} = \mathbb{N}$ is the number of measurement, a hyperparameter to be specified. Given the data set $(\bx_i, y_i)_{i=1}^n$, we solve the optimization problem $\inf_{\bb\in\R^p} \sum_{i=1}^n \ell_{y_i}(\bx_i^\top\bb)$ using the loss $\ell_{y}(t) = q\log(1+\exp(t)) - yt$; in other words, we compute the corresponding MLE.  Similarly to \Cref{fig:poisson_phase_transition}, \Cref{fig:q_logistic_phase_transition} plots the count of instances where a minimizer exists in \eqref{infimum} along with the theoretical threshold $1/\delta_\infty$ for each hyperparameter $q\in \{1, 2, 4, 8, 16\}$. When $q=1$, the simulation setting is reduced to the Binary logistic regression, thereby recovering the figure in \cite{candes2020phase}. The result for $q\ge 2$ is new, and we observe that the
generalized threshold \eqref{eq:threshold_definition} predicts well the existence of the MLE.

\section*{Impact Statement}


This paper presents work whose goal is to advance the field of 
Machine Learning. There are many potential societal consequences 
of our work, none which we feel must be specifically highlighted here.



\bibliography{reference}
\bibliographystyle{icml2025}

\newpage
\appendix
\onecolumn


\section{Derivation of threshold from convex geometry}
Define $\delta_\infty\in (0, +\infty]$ as 
$$
1/\delta_\infty \coloneq \inf_{t\in\R}\varphi(t)
$$
where $\varphi:\R\to\R$ is the convex function defined as 
$$
\varphi(t) \coloneq \E\bigl[(G+Ut)^2 I\{\Omega_\vee(Y)\}\bigr] 
+ \E\bigl[(G+Ut)_{+}^2 I\{\Omega_\nearrow(Y)\}\bigr]
+ \E\bigl[(G+Ut)_{-}^2 I\{\Omega_\searrow(Y)\}\bigr].
$$
Here $\Omega_{\vee}, \Omega_{\nearrow}, \Omega_{\searrow}$ are $\sigma(Y)$-measurable events defined in \eqref{eq:def_Omega}. 
We denote by $\mh$ the Hilbert space which consists of measurable function of $(G, Y, U)$ with finite second moments.  Let $\mathcal{C} \subset \R \times \mh$ be the cone defined as 
\begin{align}\label{eq:def_cone_Hilbert}
    (t, p)\in \mathcal{C} \Leftrightarrow tU + p 
    \begin{cases}
        \le 0  &\text{under }\Omega_{\nearrow}(Y)\\
        = 0  & \text{under }\Omega_{\vee} (Y)\\
        \ge 0  &\text{under }\Omega_{\searrow} (Y)
    \end{cases}
\end{align}
\begin{lemma}\label{lemma:threshold_exist}
The threshold $\delta_\infty$ can be represented as 
\begin{align*}
    \delta_{\infty}^{-1} = \inf_{t\in\R}\varphi(t) = \inf_{t\in\R} \E[(G-p(t))^2] = \inf_{(t, p)\in \mathcal{C}} \E[(G-p)^2]
\end{align*}
where $(t, p(t))\in \mathcal{C}$ for all $t\in\R$ and $p(t)$ is given by 
$$
p(t) \coloneq -tU + \begin{cases}
    0  &\text{under $\Omega_\vee(Y)$}\\
    (G+Ut)_{-}  &\text{under $\Omega_\nearrow(Y)$}\\
    (G+Ut)_+ & \text{under $\Omega_\searrow(Y)$}. 
\end{cases}
$$  
Suppose that the law of $(U, Y)$ satisfies 
                \begin{align}\label{eq:equivalent_condition_existence_threshold}
                    \begin{split}
                        &\E\Bigl[U^2 \bigl(I\{\Omega_\vee(Y)\} + I\{\Omega_\searrow(Y), U>0\}+ I\{\Omega_\nearrow(Y), U<0\}\bigr)\Bigr] > 0,\\
                        & \E\Bigl[U^2 \bigl(I\{\Omega_\vee(Y)\} +  I\{\Omega_\searrow(Y), U<0\} + I\{\Omega_\nearrow(Y), U>0\}\bigr)\Bigr] > 0.        
                    \end{split}
                \end{align}
         Then the map $\varphi$ is coercive, i.e., $\lim_{|t|\to+\infty}\varphi(t)=+\infty$ and 
          $\inf_{t\in\R}\varphi(t)$ admits a minimizer $t_*\in \R$. Furthermore, the optimal $p_*=p(t_*)\in\mh$  satisfies 
         $$
        \E[p_*^2]=\E[p_*G], \quad \|p_*\| = \sqrt{1-\delta_{\infty}^{-1}}. 
         $$
\end{lemma}
\begin{proof}
                Fix $t\in\R$. By the definition of the cone $\mathcal{C}$ and $\varphi$, it easily follows that 
                \begin{align*}
                    \inf_{p\in\mh: (t, p)\in \mathcal{C}} \E[(G-p)^2] = \E[(G-p(t))^2] 
                    = \varphi(t) 
                \end{align*}
                This proves the representation $\inf_{t\in\R}\varphi(t)= \inf_{(t, p)\in\mathcal{C}}\E[(G-p)^2] = \inf_{t\in\R}\E[(G-p(t))^2]$.

                Next, let us show that the map $\varphi:\R\to\R$
                \begin{align*}
                   \varphi(t) =  \E\bigl[(G+Ut)^2 I\{\Omega_\vee(Y)\}\bigr] 
                   + \E\bigl[(G+Ut)_{+}^2 I\{\Omega_\nearrow(Y)\}\bigr]
                   + \E\bigl[(G+Ut)_{-}^2 I\{\Omega_\searrow(Y)\}\bigr]
                \end{align*}
                is coercive. Since $t\mapsto \varphi(t)$ is convex, it suffices to show the coercivity, i.e., 
                $\lim_{|t|\to+\infty}\varphi(t)=+\infty$. For the first term, expanding the square of $(G+Ut)^2$, it immediately follows that
                 $$
                 \E[(G+Ut)^2 I\{\Omega_\vee(Y)\}] = t^2 \E[U^2 I\{\Omega_\vee(Y)\}] + O(t) 
                 $$
                 For the second term $\E[(G+Ut)_+^2 I\{\Omega_\nearrow(Y)\}]$, 
                 \begin{align*}
                    \E[(G+Ut)_+^2 I\{\Omega_\nearrow(Y)\}] &= \E[(G+Ut)^2 I\{\Omega_\nearrow(Y), \ G+Ut > 0\}]\\
                    &= t^2 \E[U^2 I\{\Omega_\nearrow(Y)\} I\{G+Ut > 0\}] + O(|t|)
                 \end{align*}
                 as $|t|\to+\infty$. 
                 By $\E[U^2]=1<+\infty$, the dominated convergence theorem implies 
                 \begin{align*}
                    \E[U^2 I\{\Omega_\nearrow(Y)\} I\{G+Ut > 0\}]
                    &=\E[U^2 I\{\Omega_\nearrow(Y)\} I\{G+Ut > 0, U>0\}] \\
                    &+ \E[U^2 I\{\Omega_\nearrow(Y)\} I\{G+Ut > 0, U<0\}]\\
                    &\to \begin{cases}
                        \E[U^2I\{\Omega_\nearrow(Y), U>0\}] & t\to+\infty\\
                        \E[U^2I\{\Omega_\nearrow(Y), U<0\}] & t\to-\infty
                    \end{cases}
                 \end{align*}
                 Thus, 
                 \begin{align*}
                    \frac{1}{t^2}\E[(G+Ut)_+^2 I\{\Omega_\nearrow(Y)\} \to \begin{cases}
                        \E[U^2I\{\Omega_\nearrow(Y), U>0\}] & t\to+\infty\\
                        \E[U^2I\{\Omega_\nearrow(Y), U<0\}] & t\to-\infty
                    \end{cases} 
                 \end{align*}
                 By the same argument, we have
                \begin{align*}
                    \frac{1}{t^2} \E[(G+Ut)_{-}^2 I\{\Omega_\searrow(Y)\}] \to \begin{cases}
                        \E[U^2I\{\Omega_\searrow(Y), U<0\}] & t\to+\infty\\
                        \E[U^2I\{\Omega_\searrow(Y), U>0\}] & t\to-\infty
                    \end{cases} 
                \end{align*}
               Putting them together, we obtain
                \begin{align*}
                    \frac{ \varphi(t)}{t^2} \to \begin{cases}
                        \E[U^2 I\{\Omega_\vee(Y)\}] + \E[U^2I\{\Omega_\nearrow(Y), U>0\}] + \E[U^2I\{\Omega_\searrow(Y), U<0\}] & t\to+\infty\\
                        \E[U^2 I\{\Omega_\vee(Y)\}] + \E[U^2I\{\Omega_\nearrow(Y), U<0\}] + \E[U^2I\{\Omega_\searrow(Y), U>0\}] & t\to-\infty
                    \end{cases}.
                \end{align*}
               Since the limit is strictly positive by the condition \eqref{eq:equivalent_condition_existence_threshold}, the map $t\mapsto \varphi(t)$ is coercive, and hence $\inf_{t\in\R}\varphi(t)$ is attained at some $t=t_*\in\R$.

                Finally, we prove $\E[p_*^2]=\E[p_*G]$ and $\|p_*\|=\sqrt{1-\delta_\infty^{-1}}$. The stationary condition of $\varphi'(t_*) = 0$ gives 
                \begin{align*}
                    0&= 2 \E[(G+Ut_*) U I\{\Omega_\vee(Y)\}] 
                    + 2\E[(G+Ut_*)_{+} U I\{\Omega_\nearrow(Y)\}]
                    + 2\E[(G+Ut_*)_{-} U I\{\Omega_\searrow(Y)\}]     \\
                    &= 2 \E[(G-p_*)U]    
                \end{align*}
                where 
                $$
                p_* \coloneq p(t_*) = -t_*U +(G+Ut_*)_- I\{\Omega_\nearrow(Y)\} + (G+Ut_*)_+ I\{\Omega_\searrow(Y)\}
                $$
                Here, $\E[GU]=0$ since $G$ and $U$ are independent standard normal. Thus, the last equation gives $\E[p_*U] = 0$. Then, we have 
            \begin{align*}
                \E[p_*^2] &= -t_*^2 + t_*^2 + \E[p_*^2] + 2t_* \E[p_* U] && \E[p_* U]=0 \\
                &= -t_*^2 + \E[(p_*+Ut_*)^2]\\
                &=  -t_*^2 + \E\Bigl[I\{\Omega_\nearrow(Y)\}(G+Ut_*)_{-}^2\Bigr] +  \E\Bigl[I\{\Omega_\searrow(Y)\}(G+Ut_*)_{+}^2\Bigr],\\
                \E[p_* G] &= \E[(p_* + Ut_*) G] && \E[UG]=0\\
                &= \E[(t_*U+G)_{-} G I\{\Omega_\nearrow(Y)\}] +  \E[(t_*U+G)_{+} G I\{\Omega_\searrow(Y)\}]
            \end{align*}
            By $(G+t_*U)_{\pm}^2=(G+t_* U)_{\pm}(G+t_*U)$ and the definition of $p_*$, we get 
            \begin{align*}
                &\E[p_*^2]-\E[p_*G] \\
                &= -t_*^2  + \E[(t_*U+G)_{-} t_*U I\{\Omega_\nearrow(Y)\}] +  \E[(t_*U+G)_{+} t_*U I\{\Omega_\searrow(Y)\}]\\
                &= -t_*^2 + t_*\E\Bigl[U \bigl((G+Ut_*)_- I\{\Omega_\nearrow(Y)\} + (G+Ut_*)_+ I\{\Omega_\searrow(Y)\}\bigr)\Bigr] \\
                &= -t_*^2 + t_* \E[U (t_*U+p_*) ]&& \text{by the definition of $p_*$} \\
                &= -t_*^2 + t_*^2 + 0 && \E[U^2]=1, \ \E[Up_*]=0\\
                &=0.
            \end{align*}
            The equation $\|p_*\|=\sqrt{1-\delta_{\infty}^{-1}}$ follows from $\|p_*\|^2 = \E[p_*G]$ and $\delta_{\infty}^{-1}=\E[(G-p_*)^2]$. 
\end{proof}

\subsection{Proof of \Cref{theorem:phase_transition_conic_geometry}}
\begin{lemma}\label{lemma:equivalent_condition_mle}
    Suppose that $\ell_{y_i}$ is strictly convex, $C^1$, and $\inf_u \ell_{y_i}(u)$ is finite. Then the M-estimator does not exist if and only if
\begin{align}\label{eq:equivalent_condition_mle}
    \exists \bb_{*} \in\R^p\setminus\{0\} \ \text{such that} \left(
        \forall i\in [n], \quad \bx_i^\top\bb_{*}  \begin{cases}
        \ge 0 & \text{under } \Omega_{\searrow}(y_i)\\
        = 0 & \text{under } \Omega_{\vee}(y_i)\\
        \le 0 &\text{under } \Omega_{\nearrow}(y_i)\\
    \end{cases} 
    \right)
\end{align}
\end{lemma}
\begin{proof}
    \textbf{\eqref{eq:equivalent_condition_mle} holds $\Rightarrow$ M-estimator does not exist:}. 
    Let $L(\bb) \coloneq \sum_{i=1}^n \ell_{y_i} (\bx_i^\top\bb)$ be the objective function. 
    Suppose there exists $\bb_*\in \R^p\setminus\{\bm{0}_p\}$ such that \eqref{eq:equivalent_condition_mle} is satisfied. Then, the map $\R_{\ge 0}\ni t\mapsto L(t\bb_*)$ is uniformly bounded by $\sum_{i=1}^n \ell_{y_i}(0)$. For all $\bb\in\R^p$, let $\bb_\nu $ be the convex combination
    $$
   \bb_\nu = (1-1/\nu) \bb + (1/\nu) (\nu \bb_*), \quad \nu > 1. 
    $$
    Note $\bb_\nu\to\bb + \bb_*$ as $\nu\to+\infty$. 
    By the convexity of $L(\bb)$ and uniform bound $\sup_{t\ge 0} L(t\bb_*) \le \sum_{i=1}\ell_{y_i}(0)$, we have
    $$  
    L(\bb_\nu ) \le (1-1/\nu) L(\bb_*)  + (1/\nu) L(\nu \bb_*) \le (1-1/\nu) L(\bb_*)  + (1/\nu) \sum_{i=1}^n \ell_{y_i}(0)
    $$
    Taking $\nu\to+\infty$, the RHS converges to $L(\bb_*)$, while the LHS converges to $L(\bb+\bb_*)$ by the continuity of $L(\cdot)$ and $\bb_\nu\to \bb+\bb_*$, so we are left with
    $$
    L(\bb+\bb_*)\le L(\bb), \quad \forall \bb\in\R^p. 
    $$
    Now we suppose that the M-estimator $\hat{\bb}\in \argmin_{\bb\in\R^p}L(\bb)$ exists. Setting $\bb=\hat{\bb}$ in the above display, since $\bb_*\ne \bm 0_p$, we know that $\tilde{\bb} \coloneq \bb_* + \hat{\bb}$ is also a minimizer with $\tilde{\bb}\ne \hat{\bb}$. For all $t\in (0,1)$, by the convexity of $L(\cdot)$, we have
    $$
    L(t\hat{\bb} + (1-t)\tilde{\bb}) \le t L(\hat{\bb}) + (1-t) L (\hat{\bb})
    $$
    Since $\hat{\bb}$ and $\tilde{\bb}$ minimize $L$, this holds in equality. With $L(b)=\sum_{i}\ell_{y_i} (\bx_i^\top b)$ and by the convexity of $\ell_{y_i}$, the equality condition reads to  
    $$
    \forall i \in [n], \quad \ell_{y_i}(t \bx_i^\top \hat{\bb} + (1-t) \bx_i^\top \tilde{\bb}) = t \ell_{y_i}(\bx_i^\top \hat{\bb}) + (1-t)\ell_{y_i}(\bx_i^\top \tilde{\bb})
    $$
    for all $t\in (0,1)$. 
    By the strict convexity of $\ell_{y_i}$, we must have $\bx_i^\top \hat{\bb} = \bx_i^\top \tilde{\bb}$ for all $i\in[n]$, i.e., $\hat{\bb}-\tilde{\bb}\in \operatorname{Ker}(\bm{X})$. However, since $\bm{X}\in\R^{n\times p}$ has iid $N(0,1)$ entry, $\bm{X}$ is an $n\times p$ matrix with  $\operatorname{rank}(X)=p<n$. This implies $\operatorname{Ker}(\bm{X})=\{\bm{0}_p\}$ and $\hat{\bb}-\tilde{\bb}=\bm{0}_p$, which is a contradiction with $\hat{\bb}\ne \tilde{\bb}$. Thus, if \eqref{eq:equivalent_condition_mle} holds then the M-estimator does not exist. \\
    
    \textbf{M-estimator does not exist $\Rightarrow$ \eqref{eq:equivalent_condition_mle} holds:}
    Suppose that the M-estimator does not exist. Then, there exists a sequence $(\bb_k)_{k=1}^\infty$ such that as $k\to+\infty$, we have $L(\bb_k)\to \inf_{\bb\in\R^p}L(\bb)$ and $\|\bb_k\|\to+\infty$. 
By the compactness of the unit sphere in $\R^p$, we can extract a subsequence $(\bb'_k)_{k=1}^\infty$ of $(\bb_k)_{k=1}^\infty$ such that $\bb_k'/\|\bb_k'\|\to \bv$ for a unit sphere vector $\bv\in\R^p$. Therefore, we can assume without loss of generality that \( \bb_k / \|\bb_k\| \) converges to a unit sphere vector \( \bv \).

We proceed by contradiction; suppose \eqref{eq:equivalent_condition_mle} is not satisfied. Then we can find an index \(i=i(\bv)\in [n] \) associated with the unit sphere vector $\bv$ such that 
\begin{align*}
       \bx_i^\top\bv \begin{cases}
        < 0 & \text{under } \Omega_{\searrow}(y_i)\\
        \ne 0 & \text{under } \Omega_{\vee}(y_i)\\
        > 0 &\text{under } \Omega_{\nearrow}(y_i)\\
    \end{cases} 
\end{align*}
If \( \Omega_\searrow(y_i) \), or \( \Omega_\nearrow(y_i) \), then the derivative of \( t \mapsto \ell_{y_i} (t \bx_i^\top \bv) \) is positive for all \( t \geq 1 \).  
If \( \Omega_\vee \), it is not necessarily positive right away at \( t=1 \), but eventually positive for \( t \) large enough, say some \( t_*\): for any \( t > t_* \), the derivative of \( t \mapsto \ell_i(t \bx_i^\top \bv) \) is positive.
Call this derivative at \( t_* \), say \( A = \bx_i^\top \bv \ell_{y_i}'(t_* \bx_i^\top \bv) \) and \( A > 0 \).

Let \( \bv_k = \bb_k / \|\bb_k\| \) so that $\bx_i^\top \bv_k \to \bx_i^\top \bv$ . Since $\ell_{y_i}$ is $C^1$,  
for \( k \) large enough we have that the derivative of \( t \mapsto \ell_{y_i}(t \bx_i^\top \bv_k) \) at \( t_* \) is larger than \( A / 2 \).  
Call this derivative \( A_k = \bx_i^\top \bv_k \ell_{y_i}'(t_* \bx_i^\top \bv_k) \) so that \( A_k > A / 2 > 0 \).

By the convexity of $t\mapsto \ell_{y_i}(t \bx_i^\top \bv_k)$, we have 
\[
\ell_{y_i}(\bx_i^\top \bb_k) =\ell_{y_i} (\|\bb_k\|_2 \bx_i^\top\bv_k)
\geq \ell_{y_i}(t_* \bx_i^\top \bv_k) + (\|\bb\|_k - t_*) A_k  
\geq \ell_{y_i}(t_* \bx_i^\top \bv_k) + (\|\bb_k\| - t_*) A/2.
\]
This gives 
\[
L(\bb_k) = \sum_{j=1}^n 
\ell_{y_j}(\bx_j^\top \bb_k) \geq \sum_{j \ne i} \inf_u \ell_{y_j}(u)+ \ell_{y_i}(t_* \bx_i^\top \bv_k) + (\|\bb_k\| - t_*) A,
\]
where the RHS goes to \( \infty \) since $\|\bb_k\|\to+\infty$ and $A>0$. This is a contradiction with $L(\bb_k)\to \inf_{\bb}L(\bb)<+\infty$. Thus, if the M-estimator does not exist then \eqref{eq:equivalent_condition_mle} holds. 
\end{proof}

\begin{proof}[Proof of \Cref{theorem:phase_transition_conic_geometry}]
By the rotational invariance of $\bx_i\sim N(\bm{0}_p, I_p)$, we can assume without loss of generality that $y_{i}$ depend on $\bx_i = (x_{i1}, x_{i2}, \dots, x_{ip})\in\R^p$ through its first coordinate $x_{i1}$, i.e., 
\begin{align}\label{eq:independent_cordinate}
    (y_i, x_{i1}) \indep (x_{i2}, \cdots x_{ip}).    
\end{align}
By \Cref{lemma:equivalent_condition_mle} the M-estimator does not exist if and only if
{
    $$
    \operatorname{Span}\bigl(X \bm{e}_2, \dots, X\bm{e}_p \bigr) \cap \mathcal{C}\bigl(X\bm{e}_1, \by\bigr) \ne \{\bm{0}_n\}, 
    $$
    where $\bm{e}_{i} (i=1, 2, \dots, p)$ are canonical basis vector in $\R^p$ and $\mathcal{C}\bigl(X\bm{e}_1, \by\bigr)$ is the cone in $\R^n$ defined as follows:
    \begin{align*}
        \forall \bu, \by\in \R^n, \quad 
        \bm{p} \in \mathcal{C} (\bu, \by) \Leftrightarrow \exists t\in\R \text{ such that } \left( \forall i\in [n], \ 
    t u_i + p_i = \begin{cases}
        \ge 0 & \text{under } \Omega_{\searrow}(y_i)\\
        = 0 & \text{under } \Omega_{\vee}(y_i)\\
        \le 0 &\text{under } \Omega_{\nearrow}(y_i)\\
    \end{cases} 
   \right). 
    \end{align*}
    Here, $\operatorname{Span}\bigl(X \bm{e}_2, \dots, X\bm{e}_p \bigr)$ is the linear space spanned by the $(p-1)$ vectors $(X \bm{e}_i)_{i=2}^p$, which is a rotationally invariant random subspace in $\R^{n}$ with dimension $p-1$ since $X \bm{e}_i$ are independent standard normal. Furthermore, \eqref{eq:independent_cordinate} implies that $\operatorname{Span}\bigl(X \bm{e}_2, \dots, X\bm{e}_p \bigr)$ is independent of the cone $\mathcal{C}(X\bm{e}_1, \by)$. Below, we write $X\bm{e}_1=\bu$ for simplicity.  
}
Then, by Theorem I of \citet{amelunxen2014living}
we have that for \( \eta\in (0,1/2) \), conditionally on \(\bu, \by\), 
\begin{equation}\label{eq:finite_sample_phase_transition}
\begin{aligned}
    p-1 + \statdim(\mathcal{C}(\bu, \by)) \ge n + n^{1/2+\eta} & \Rightarrow
    \PP(\text{M-estimator exists} \mid \bu,\by) \to 0,
    \\
    p-1 + \statdim(\mathcal{C}(\bu, \by)) \le n - n^{1/2+\eta} & \Rightarrow
    \PP(\text{M-estimator exists} \mid \bu,\by) \to 1
\end{aligned}
\end{equation}
almost everywhere. Here $\statdim(\mathcal{C}(\bu, \by))$ is the statistical dimension of the cone given by
$$
\statdim(\mathcal{C}(\bu, \by)) = n-\E[\dist(\bm{g}, \mathcal{C}(\bu, \by))^2\mid X\bm{e}_1, \by ] \quad \text{with} \quad  \bg\sim N(\bm{0}_n, I_n),
$$
where the expectation is taken with respect to $\bg$, which is independent of $(\bu, \by)$. Substituting this to \eqref{eq:finite_sample_phase_transition} with $\eta$ set to $1/4$, we get 
\begin{align*}
  p-1 \ge \E[\dist(\bm{g}, \mathcal{C}(\bu, \by))^2\mid\bu, \by ] + n^{3/4} & \Rightarrow
    \PP(\text{M-estimator exists} \mid \bu,\by) \to 0\\
    p-1 \le  \E[\dist(\bm{g}, \mathcal{C}(\bu, \by))^2\mid\bu, \by ] -  n^{3/4} & \Rightarrow
    \PP(\text{M-estimator exists} \mid \bu,\by) \to 1
\end{align*}
With $p/n\to\delta^{-1}$, if we prove the convergence
$$
n^{-1} \E[\dist(\bm{g}, \mathcal{C}(\bu, \by))^2 \mid
\bu, \by ] \to^p  \delta_\infty^{-1}
$$
then we complete the proof. Below we prove this. Recall $\bu=X\bm{e}_1$ so that $(\bu, \by) = (u_i, y_i)_{i=1}^n \iid (U, Y)$. 
By the explicit gradient identities (B.7)-(B.9) in \cite{amelunxen2014living},
the Euclidean norm of the gradient of 
\(\bm g \mapsto \text{dist}(\bg, \mathcal{C})^2 \) is bounded by \( 2\|\bm g\|_2 \). Thus, conditionally on \( \bu, \by \), the Gaussian Poincar\'e inequality (cf. Theorem 3.20\cite{boucheron2013concentration}) yields
$$
\E[\dist(\bm{g}, \mathcal{C}(\bu, \by))^2|\bu, \by ] = \dist(\bm{g}, \mathcal{C}(\bu, \by))^2 + O_P(\sqrt{n}). 
$$
Here $\dist(\bm{g}, \mathcal{C}(\bu, \by))^2 = \inf_{\bm{p}\in \mathcal{C}(\bu, \by)} \|\bg - \bp\|^2$ is equal to the optimal value of 
\begin{align*}
   \inf_{(t, \bp)\in \R\times \R^n} \sum_{i=1}^n (g_i-p_i)^2  \quad \text{subject to} \quad  \left( \forall i\in [n], \ 
    tu_i + p_i = \begin{cases}
        \ge 0 & \text{under } \Omega_{\searrow}(y_i)\\
        = 0 & \text{under } \Omega_{\vee}(y_i)\\
        \le 0 &\text{under } \Omega_{\nearrow}(y_i)\\
    \end{cases} 
   \right)
\end{align*}
For each $t$, we can solve the minimization with respect to $\bp\in\R^n$. The optimal $\bp=\bp(t)$ is given by 
$$
p_i(t) = -t u_i + \begin{cases}
    (g_i+u_i t)_+ & \text{under $\Omega_\searrow(y_i)$}\\
    0  &\text{under $\Omega_\vee(y_i)$}\\
    (g_i+u_i t)_{-}  &\text{under $\Omega_\nearrow(y_i)$}.
\end{cases}
$$
Therefore, $\frac{1}{n}\dist(\bm{g}, \mathcal{C}(\bu, \by))^2$ can be written as 
\begin{align*}
    \frac{1}{n}\dist(\bm{g}, \mathcal{C}(\bu, \by))^2  &= \frac{1}{n}\inf_{t\in\R} \sum_{i=1}^n  (t u_i + g_i)_-^2 I\{\Omega_\searrow(y_i)\} + (t u_i + g_i)^2 I\{\Omega_\vee(y_i)\} + (t u_i + g_i)_{+}^2 I\{\Omega_\nearrow(y_i)\}\\ 
    &= \inf_{t\in\R} \varphi_n(t), 
\end{align*}
where
$$
\varphi_n(t) \coloneq  \frac{1}{n} \sum_{i=1}^n  (t u_i + g_i)_-^2 I\{\Omega_\searrow(y_i)\} + (t u_i + g_i)^2 I\{\Omega_\vee(y_i)\} + (t u_i + g_i)_{+}^2 I\{\Omega_\nearrow(y_i)\}.
$$
Notice that $\varphi_n(t)$ is a random and convex function, and by the law of large number, 
$$
\varphi_n(t)\to^p \varphi(t) = \E[(t U + G)_-^2 I\{\Omega_\searrow(Y)\}] + \E[(t U + G)^2 I\{\Omega_\vee(Y)\}] + \E[(t U + G)_{+}^2 I\{\Omega_\nearrow(Y)\}]
$$
for each $t\in\R$. By \Cref{lemma:threshold_exist}, we have 
$$
\delta_\infty^{-1} = \inf_{t\in\R} \varphi(t),
$$
and $\varphi(t)$ is coercive, i.e., $\lim_{|t|\to+\infty} \varphi(t)=+\infty$. Then, $\inf_{t\in\R}\varphi(t)$ can be reduced to $\min_{t\in K}\varphi(t)$ for a compact set $K\subset\R$, and if a convex function converges point-wisely then it converges uniformly over any compact set. 
This provides $\inf_{\in\R} \varphi_n (t)\to^P \inf_{t\in\R} \varphi(t) = \delta_\infty^{-1}$ and completes the proof.  
\end{proof}

\section{Set up for infinite-dimensional optimization problem}
\begin{lemma}\label{lm:objective_func_property}
   Suppose that $\ell_Y:\R\to \R$ is a proper lower semicontinuous convex function. Then the map
    $$
    \ml:  \R\times \mh \to  \R\cup\{+\infty\} , \quad (a, v) \mapsto \begin{cases}
        \E[\ell_Y(aU+v)] & \text{if } \E[|\ell_Y(aU+v)|] <+\infty\\
        +\infty  &\text{otherwise}
    \end{cases}
    $$
    is again a proper lower semicontinuous convex function. Furthermore, for all $(a, v)\in \operatorname{dom}\ml$, the subderivative at $(a, v)$ is given by
    \begin{align*}
        \partial_a \ml &= \bigl\{\E[U h]: h\in \partial \ell_Y(aU+v)\bigr\} \cap \R, \\
        \partial_v \ml &= \partial \ell_Y(aU+v) \cap \mh. 
    \end{align*}
\end{lemma}
\begin{proof}
    Proposition 16.63 in \citet{bauschke2017convex}. 
\end{proof}

\begin{lemma}[{Lemma A.1 of \citet{bellec2023existence}}]\label{lm:constraint_property}
    Define $\mg:\mh\to\R$ as 
    $$
    \mg:\mh \to \R, \quad v\mapsto \|v\|-\frac{\E[vG]}{\sqrt{1-\delta^{-1}}}.
    $$
    Then $\mg$ is convex, Lipschitz, and finite valued. Furthermore, 
    $\mg$ is Fr\'echet differentiable at $\mh \setminus \{0\}$ in the sense that 
    $
    \mg(v+h) = \mg(v) + \E[\nabla\mg(v) h] + o(\|h\|)
    $ for all $\|v\|>0$, 
    where the gradient $\nabla \mg(v)$ is given by 
    $$
    \nabla\mg(v) = \frac{v}{\|v\|} - \frac{G}{\sqrt{1-\delta^{-1}}}.
    $$
\end{lemma}

\begin{lemma}[Existence of Lagrange multiplier]\label{lagrange}
    Assume $\E[|\ell_Y(G)|]<+\infty$. Then $(a_*, v_*)\in \operatorname{dom}\ml$ solves the constrained optimization problem:
    $$
    \min_{(a, v)\in \R\times \mh} \ml(a, v) \quad \text{subject to} \quad \mg(v)\le 0
    $$
    if and only if there exists an associated Lagrange multiplier $\mu_*\ge 0$ such that the KKT condition below is satisfied: 
    \begin{align}\label{eq:KKT}
        -\mu_* \partial \mg(v_*)\cap \partial_v \ml(a_*, v_*)\ne \emptyset, \quad 0 \in \partial_a \ml(a_*, v_*), \quad \mu_* \mg(v_*) = 0, \quad \mg(v_*)\le 0. 
    \end{align}
    Now we further assume that
    \begin{itemize}
        \item $\E[|\inf_u \ell_Y(u)|] <+\infty$.  
        \item $\ell_Y$ is strictly convex. 
        \item $\PP(\Omega_\searrow) + \PP(\Omega_\nearrow)>0$ and $\ell_Y(\cdot)$ is not constant with probability $1$. 
        \item There exists a positive constant $b$ and a positive random variable $D(Y)$ such that under $\Omega_{\vee}$, 
        $$
        \ell_Y(u) \ge b^{-1}|u| - D(Y), \quad \forall u\in \R
        $$
        with $\E[D(Y)^2  I\{\Omega_\vee\}]<+\infty$. 
    \end{itemize}
    Then, the Lagrange multiplier is always strictly positive $\mu_*>0$ and the constraint is binding, i.e., $\mathcal{G}(v_*)=0$. 
\end{lemma}
\begin{proof}
    First we verify Slater's condition:
$$
\operatorname{lev}_{\leq 0} \mg \subseteq \operatorname{int} \operatorname{dom} \mg, \quad 
\operatorname{dom} \ml \cap  \operatorname{lev}_{< 0} \mg \neq \emptyset. 
$$
    Since $\mg(v)=\|v\|-\E[vG]/\sqrt{1-\delta^{-1}}$ is finite valued the first condition $\operatorname{lev}_{\leq 0} \mg \subseteq \operatorname{int} \operatorname{dom} \mg$ immediately follows. As for the second condition, $(a, v) = (0, G)$ satisfies $\mg(G) = 1-(1-\delta^{-1})^{-1/2} <0$ and $|\ml(a, v)| = |\E[\ell_Y(G)]|\le \E[|\ell_Y(G)|]<+\infty$ by the assumption. Therefore, the objective function and the constraint $(\ml, \mg)$ satisfy Slater's condition. 
   With Slater' condition, the ``if and only if ~'' part follows from Proposition 27.21 of \citet{bauschke2017convex}. 

Let us show $\mu_*>0$. Suppose $\mu_*=0$. Then, $(a_*, v_*)$ solves $\min_{(a,v)}\E[\ell_Y(aU+v)]$.
   Now, for any $n\in\mathbb{N}$, define $v_n \in\mh$ as 
    \begin{align*}
        v_n = \begin{cases}
          u_{\min} I\{\ell_Y(u_{\min})\le n\}    & \text{under } \Omega_\vee\\
            -n & \text{under } \Omega_\nearrow\\
            n & \text{under } \Omega_\searrow
        \end{cases} \quad \text{where } u_{\min}\in \argmin_{u\in\R} \ell_Y(u)
    \end{align*}
    Note that $v_n$ is in $\mh$ since the coercivity assumption implies that under the event  $\Omega_{\vee}$,
    $$
    |u_{\min}| I\{\ell_Y(u_{\min})\le n\} \le b (n + D(Y)) I\{\ell_Y(u_{\min})\le n\}
    $$
    and the RHS has a finite second moment under $\Omega_{\vee}$ by the moment assumption on $D(Y)$. 
 Evaluating the objective function $\ml$ at $(a, v)=(0, v_n)$, by the optimality of $(a_*, v_*)$, we are left with
 \begin{align*}
    \E[\ell_Y(a_*U+v_*)] &\le \E[\ell_Y(0\cdot U + v_n)]\\
    &= \E[\ell_Y(u_{\min}) I\{\ell_Y(u_{\min}) \le n\}I\{ \Omega_\vee\}] \\
    &+  \E[\ell_Y(0) I\{\ell_Y(u_{\min}) > n\}I\{ \Omega_\vee\}] \\
    &+ \E[\ell_Y(-n) I\{\Omega_\nearrow\}] \\
    &+ \E[\ell_Y(n) I\{\Omega_\searrow\}].
 \end{align*}
 Note that each integrand on RHS is uniformly bounded by $|\inf_u \ell_Y(u)| + |\ell_Y(0)|$, where $|\inf_u \ell_Y(u)|$ and $|\ell_Y(0)|$  have finite moments by the assumption and \Cref{lm:ell_0_bounded}. Thus, by the dominated convergence theorem, taking the limit $n\to+\infty$, we obtain
 \begin{align*}
    \E[\ell_Y(a_* U + v_*)] &\le \E[\min_u \ell_Y(u)I\{\Omega_\vee\}] + 0 + \E[\inf_u \ell_Y(u)I\{\Omega_\nearrow\}] + \E[\inf_u \ell_Y(u)I\{\Omega_\searrow\}] \\
    &= \E[\inf_u \ell_Y(u)]
 \end{align*}
and $\E[\ell_Y(a_* U + v_*)-\inf_u\ell_Y(u)] \le 0$. Since the integrand $\ell_Y(a_* U + v_*)-\inf_u\ell_Y(u)$ is non-negative, we get
$$
\ell_Y(a_* U + v_*)=\inf_u\ell_Y(u)
$$
with probability $1$. Let us consider the event $\Omega_{\nearrow}$. By the strict convexity of $\ell_Y(\cdot)$, under the event $\Omega_{\nearrow}$, we have always $\ell_Y(x)>\lim_{t\to+\infty}\ell_Y(-t) = \inf_u \ell_Y(u)$ for all $x$. This means that $\ell_Y(a_* U + v_*)=\inf_u\ell_Y(u)$ cannot occur under $\Omega_{\nearrow}$, and hence $\PP(\Omega_{\nearrow})=0$. By the same argument, we get  $\PP(\Omega_{\searrow})=0$. This is a contradiction with $\PP(\Omega_{\nearrow})+\PP(\Omega_{\searrow})>0$, so we must have $\mu_*>0$.


\end{proof}

\begin{lemma}\label{lm:ell_0_bounded}
    Suppose $\E[\ell_Y(G)_+] < +\infty$ and $\E[\inf_u \ell_Y(u)_{-}] > - \infty$ where $G\sim N(0,1)$ independent of $Y$. Then, $\E[|\ell_Y(0)|]<+\infty$.
\end{lemma}
\begin{proof}
Note 
$$
|\ell_Y(0)| \le \max(-(\inf_{u}\ell_Y(u))_{-}, \ell_Y(0)_+) \le -(\inf_u \ell_Y(u))_{-} + \ell_Y(0)_+
$$
and the RHS has a finite moment by the assumption and Jensen's inequality $\E[\ell_Y(0)_+]  = \E[\ell_Y(\E[G])_+] \le \E[\ell_Y(G)_+] <+\infty$. Here we have used the fact that $u \mapsto (\ell_Y(u))_{+}$ is convex and $G\sim N(0,1)$ is independent of $Y$. 
\end{proof}

\section{Equivalence between nonlinear system and infinite-dimensional optimization problem: \Cref{theorem:equivalence_system_optimization}}

\begin{lemma}\label{lm:prox_in_h}
    Suppose $\E[
        |\inf_u\ell_Y(u)
        |] <+\infty$ and $\E[|\ell_Y(G)]<+\infty$. Then for any $a, \sigma, \gamma\in \R \times \R_{>0}\times \R_{>0}$, 
        $\prox[\gamma \ell_Y](a U+\sigma G) \in \mh$. 
\end{lemma}
\begin{proof}
Denote $\prox[\gamma \ell_Y](aU+\sigma G)$ by $p_*$. 
Since $p_*$ minimizes $u\mapsto \frac{1}{2\gamma} (aU+\sigma G -u)^2 +  \ell_Y(u)$, we have the upper estimate
$$
\frac{1}{2\gamma} (aU + \sigma G - p_*)^2 + \ell_Y(p_*) \le \frac{1}{2\gamma} (aU+\sigma G)^2 + \ell_Y(0).
$$
With $\ell_Y(p)\ge \inf_u \ell_Y(u)$, this implies 
$$
(aU+\sigma G - p_*)^2 \le (a U+\sigma G)^2 + 2\gamma(\ell_Y(0)-\inf_u \ell_Y(u)),
$$
where the RHS has a finite moment. This means $aU+\sigma G - p_* \in \mh$ and $p_*\in \mh$.
\end{proof}

\begin{lemma}\label{lm:optimization_to_system}
    Suppose  $(a_*, v_*)\in\R\times \mh$ solves the optimization problem with $\|v_*\|>0$. Let us take a Lagrange multiplier $\mu_*>0$ satisfying the KKT condition \eqref{eq:KKT}. Define the positive scalar $(\gamma_*, \sigma_*)\in\R_{>0}^2$ by 
    $$
    \sigma_* = \frac{\|v_*\|}{\sqrt{1-\delta^{-1}}}, \quad \gamma_* = \frac{\sigma_* \sqrt{1-\delta^{-1}}}{\mu_*}.
    $$
    Then $v_*$ takes the form of
    $$
    v_* = \prox[\gamma_* \ell_Y](a_*U+\sigma_*G)-a_*U
    $$
    and $(a_*, \sigma_*, \gamma_*)$ solves the nonlinear system of equations:
    \begin{align}
            \gamma^{-2} \delta^{-1} \sigma^2 &=  \E[\ell_Y'(\prox[\gamma\ell_Y](\a U+\sigma G))^2] \label{eq:system_1}\\
            0 &= \E[U \ell_Y'(\prox[\gamma \ell_Y](\a U+\sigma G))] \label{eq:system_2} \\
            \sigma(1-\delta^{-1}) &= \E[G \prox[\gamma \ell_Y](\a U + \sigma G)] \label{eq:system_3}
    \end{align}
\end{lemma}

\begin{proof}
The map $v\mapsto \mg(v)$ is Fr\'echet differentiable at $v=v_*\ne 0$ with $\nabla\mg(v_*) = \frac{v_*}{\|v_*\|} - \frac{G}{\sqrt{1-\delta^{-1}}}$ (\Cref{lm:constraint_property}), while the constraint is binding $\mg(v_*)=0$ (\Cref{lagrange}). Thus, we have 
$$
-\mu_* \nabla\mg(v_*) =  -\mu_*\Bigl(
    \frac{v_*}{\|v_*\|} - \frac{G}{\sqrt{1-\delta^{-1}}}
    \Bigr) \in \partial_{v}\ml(a_*, v_*), \quad 0\in \partial_{a}\ml(a_*, v_*), \quad \mg(v_*)=0, 
$$
 By the definition of $(\sigma_*, \gamma_*)$, the condition $-\mu_* \nabla\mg(v_*) \in \partial_{v}\ml(a_*, v_*)$ yields 
\begin{align*}
   \partial_v \ml(a_*, v_*) \ni -(v_*-\sigma_*G)/{\gamma_*}
\end{align*}
This means that $v_*$ also minimizes the map
$$
\mh\ni v\mapsto \ml(v) + \E\Bigl[\frac{(\sigma_*G-v)^2}{2\gamma_*}\Bigr] = \E\Bigl[\ell_Y(a_*U+v) + \frac{(\sigma_*G-v)^2}{2\gamma_*}\Bigr]. 
$$
Since $\prox[\gamma_* \ell_Y](a_*U+\sigma_*G)-a_*U$ minimizes the integrand and belongs to $\mh$ by \Cref{lm:prox_in_h}, we have 
$$
v_* = \prox[\gamma_* \ell_Y](a_*U+\sigma_*G)-a_*U\in\mh. 
$$
With $\sigma G - v_* = \gamma \ell_Y'(v_*+a_*U)$, this also gives $\ell_Y'(v_*+a_*U)\in \mh$. 
This in particular means $\E[U \ell_Y'(a_* U+v_*)]\in\R$. With $\partial_a \ml(a,v) = \{\E[U \ell_Y'(a U+v)]\} \cap \R$, the condition $0 \in \partial_a \ml(a_*, v_*)$ provides 
$$
0 = \E[U \ell_Y'(a_* U+v_*)]. 
$$
With $v_* = \prox[\gamma_* \ell_Y](a_*U+\sigma_*G)-a_*U$, we get \eqref{eq:system_2}.  
As for \eqref{eq:system_1} and \eqref{eq:system_3}, rearranging $\sigma_* = \|v_*\|/\sqrt{1-\delta^{-1}}$ and $\mg(v_*)=\|v_*\|-\E[v_*G]/\sqrt{1-\delta^{-1}} = 0$ yields
\begin{align*}
    \|\sigma_* G- v_*\|^2 &= \sigma_*^2 - 2\sigma_* \E[v_*G] + \|v_*\|^2\\
    &=\sigma_*^2 -2\sigma_*^2(1-\delta^{-1}) + \sigma_*^2 (1-\delta^{-1})\\
    &= \delta^{-1} \sigma_*^2, \\
    \E[G(v_*+a_*U)] &= \sigma_*(1-\delta^{-1}) + 0.
\end{align*}
Substituting $v_* = \prox[\gamma_* \ell_Y](a_*U+\sigma_*G)-a_*U$ to these two equations, we obtain \eqref{eq:system_1} and \eqref{eq:system_3}. This completes the proof. 
\end{proof}

\begin{lemma}\label{lm:system_to_optimization}
    Suppose $(a_*, \sigma_*, \gamma_*)\in \R\times \R_{>0}^2$ solves the nonlinear system \eqref{eq:system_1}-\eqref{eq:system_3}. Then, $(a_*, v_*)\in\R\times \mh$ with 
    $$
    v_* = \prox[\gamma_* \ell_Y](a_* U +\sigma_* G) - a_*U\in\mh
    $$
    solves the infinite dimensional optimization problem with $\|v_*\|=\sigma_*\sqrt{1-\delta^{-1}}>0$, and the KKT condition \eqref{eq:KKT} is satisfied by the Lagrange multiplier $\mu_* = \sigma_*\sqrt{1-\delta^{-1}}/\gamma_* > 0$
\end{lemma}

\begin{proof} 
    We know from \Cref{lm:prox_in_h} that $v_* \in \mh$ and $\sigma G - v_* = \gamma \ell_Y'(v_*+a_*U)\in \mh$. In this case the subderivaitve of $\ml$ at $(a_*, v_*)$ is
    $$
    \partial_a \ml(a_*, v_*) = \{\E[U\ell_Y'(a_*+v_*)]\}, \quad  \partial_v \ml(a_*, v_*) = \{\ell_Y'(a_*+v_*)\}. 
    $$
    Noting $\E[UG]=0$, the nonlinear system can be written as 
        \begin{align*}
         \delta^{-1} \sigma_*^2 &= \E[(\sigma G - v_*)^2]\\
                0 &= \E[U \ell_Y'(a_*U+v_*)]\\
                \sigma_*(1-\delta^{-1}) &= \E[G v_*].
        \end{align*}
       Here, the second equation gives 
        $$
        0 \in \partial_a \ml(a_*, v_*).
        $$
       Rearranging the first and the third  equations, we have 
       $$
        \|v_*\|^2 = (1-\delta^{-1})\sigma_*^2, \quad \E[v_*G] = \sigma_*(1-\delta^{-1}).  
       $$
       This implies $\|v_*\|>0$ and 
       $\mg(v_*)=0$. Since $\|v_*\|>0$, $\mg$ is differentiable at $v_*$. The derivative formula gives 
       \begin{align*}
        -\frac{\|v_*\|}{\gamma_*} \cdot \nabla \mg(v_*) &= -\frac{\|v_*\|}{\gamma_*}\Bigl(\frac{v_*}{\|v_*\|} - \frac{G}{\sqrt{1-\delta^{-1}}}\Bigr) =\frac{\sigma_* G-v_*}{\gamma_*} = \ell_Y'(a_* U+\sigma_* G) \in \partial_v \ml(a_*,v_*)
       \end{align*}
       Therefore, $(a_*, v_*)$ satisfies the KKT condition \eqref{eq:KKT} with the Lagrange multiplier $\mu_* = \frac{\|v_*\|}{\gamma} > 0$, and $(a_*, v_*)$ solves the constrained optimization problem. 
\end{proof}

\section{Non-degeneracy and uniqueness}

\subsection{Proof of \Cref{lm:nonzero}}
The argument in this proof is inspired by the proof of Lemma 2.6] in \cite{bellec2023existence}. 
    Suppose $v_*=0$. Let $\mu_*>0$ be the associated Lagrange multiplier satisfying the KKT condition \eqref{eq:KKT} so that $v_*=0$ solves the unconstrained optimization problem
    $\min_{v\in\mh}\ml(a_*, v) + \mu_* \mg(v)$. With $\mg(0)=0$, this gives 
    $$
\E[\ell_Y(a_*U)] \le \E[\ell_Y(a_*U+v)] + \mu_* \bigl(\|v\| - \E[vG]/\sqrt{1-\delta^{-1}}\bigr) 
    $$
    for all $v\in\mh$.
    Multiplying the both sides by $\lambda \coloneq \sqrt{1-\delta^{-1}}/\mu_* >0$ and denoting \( f(\cdot)=\lambda (\ell_Y(a_*U+\cdot)-\ell_Y(a_*U)) \), we have
\begin{equation}\label{eq:ma_minorant}
0 \le \ma(v) \coloneq \E[f(v)] + \E[v^2]^{1/2} \sqrt{1-\delta^{-1}} - \E[vG] \quad\text{for all $v\in \mh$}. 
\end{equation}
We parametrize $v\in\mh$ as $v_t = \prox[t f](t G) \in \mh$ for all $t>0$ and show $\ma(v_t) < 0$ for sufficiently small $t>0$. 
Note that $t^{-1} (tG - v_t) \in \partial f(v_t)$ implies
$$
- f(v_t) = f(0) - f(v_t) \ge t^{-1}(tG-v_t) (0-v_t) = -G v_t + t^{-1} v_t^2.
$$
Substituting this to \eqref{eq:ma_minorant}, 
noting that $\E[v_tG]$ is cancelled out, we have 
\begin{align}\label{eq:ma_upper_bound}
\ma(v_t) &\le \E[Gv_t - t^{-1}v_t^2] + \E[v_t^2]^{1/2} \sqrt{1-\delta^{-1}} - \E[v_tG] \le - t \cdot \|t^{-1} v_t\| \bigl(
\|t^{-1} v_t\| - \sqrt{1-\delta^{-1}} 
\bigr).
\end{align}
Now we identify the limit of $\|v_t/t\|$ as $t\to 0+$. The Moreau envelope constructed function
\( t\mapsto \env_f(tG, t) = \frac{1}{2t}(tG-v_t)^2 + f(v_t)\)
has the derivative
\[
-\frac{1}{2 t^2}(tG - v_t)^2 + \frac{1}{t}G (tG - v_t)
= \frac 1 2 \Bigl[G^2 - \Bigl(\frac{v_t}{t} \Bigr)^2\Bigr],
\]
which is increasing in \( t \) because the Moreau envelope $\env_f(x,y)$ is jointly convex in $(x, y)\in \R\times \R_{>0}$ (cf. Lemma D.1 of \citet{thrampoulidis2018precise}). 
This means that \( v_t^2/t^2 \) is non-increasing in \( t \) and has a non-negative limit as \( t\to 0+ \). By the monotone convergence theorem, we get 
\begin{align}\label{eq:monotone_conv_v_t}
    \lim_{t\to 0^+}\|v_t/t\|^2 = \E[\lim_{t\to 0^+} (v_t/t)^2] \in [0, +\infty]    
\end{align}
Let us compute the limit $v_t/t$. First, we claim $v_t\to 0$. By the optimality of $v_t=\prox[tf](tG)$ with $f(\cdot)=\lambda(\ell_Y(a_*U+\cdot)-\ell_Y(a_*U))$, we have 
$$
\frac{1}{2t} (tG-v_t)^2 + t f(v_t) \le \frac{(tG)^2}{2t} + f(0) = tG^2 + 0. 
$$
This gives 
$$
\frac{1}{2} (tG-v_t)^2 \le t^2G^2 - t \inf_x f(x) =  t^2G^2 - t (\inf_u \ell_Y(u) - \ell_Y(a_*U))
$$
Since $(G, \inf_u \ell_Y(u), \ell_Y(a_*U))$ are all bounded in $\ell_1$, they are all finite with probability $1$. This provides $\lim_{t\to 0^+} v_t =  0$. Combined with $G-v_t/t = f'(v_t) = \lambda \ell_Y'(a_*U+v_t)$, since $\ell_Y(\cdot)$ is $C^1$, we get $v_t/t\to G-\lambda \ell_Y'(a_*U)$. Substituting this limit to \eqref{eq:monotone_conv_v_t}, we obtain
$$
\|v_t/t\|^2 \to \E[(G-\lambda \ell_Y'(a_*U))^2] \in [0, +\infty].
$$
Recall that \eqref{eq:ma_minorant} and \eqref{eq:ma_upper_bound} imply
$$
0 \le \mathcal{A}(v_t)/t \le  -\|v_t/t\|(\|v_t/t\|-\sqrt{1-\delta^{-1}}), \quad \forall t>0. 
$$
This excludes the case $\E[(G-\lambda \ell_Y'(a_*U))^2]=+\infty$ otherwise the RHS converges to $-\infty$ as $t\to 0$. Thus, we must have $\E[(G-\lambda \ell_Y'(a_*U))^2]<+\infty$ and  $\ell_Y'(a_*U) \in \mh$. Expanding the square, since $G$ and $\ell_Y'(a_*U)$ is independent, we get 
$$
\lim_{t\to+\infty} \|v_t/t\|^2  = 1 + \lambda^2 \E[\ell_Y'(a_*U)^2].
$$
Substituting this to the upper bound of $\ma(v_t)/t$,
\begin{align*}
    \lim_{t\to 0} -\|v_t/t\|(\|v_t/t\|-\sqrt{1-\delta^{-1}}) &= - \sqrt{(1 + \lambda^2 \E[\ell_Y'(a_*U)^2])} (\sqrt{1 + \lambda^2 \E[\ell_Y'(a_*U)^2]}-\sqrt{1-\delta^{-1}})\\
    & \le - (1-\sqrt{1-\delta^{-1}}) < 0,
\end{align*}
which implies that there exists a sufficiently small $t'$ such that $\ma(v_{t'})/t' <0$. This is a contradiction with $\ma(v_t)/t\ge 0$ for all $t>0$.
Therefore, we must have $v_*\ne 0$.  

\subsection{Proof of \Cref{lm:unique}}

Suppose that there exists two minimizer $(a, v), (a', v')\in \R\times \mh$. Let $(a_t, v_t)=t(a, v) + (1-t)(a', v')$ for all $t\in [0,1]$. Then, by the convexity of objective function and constraint, $(a_t, v_t)$ solves the constrained optimization problem. This implies $\ml(a_t, v_t)=t\ml(a,v)+(1-t)\ml(a', v')$ for all $t\in [0,1]$. With $\ml(a, v)=\E[\ell_Y(aU+v)]$, we must have  
\begin{align*}
\E\Bigl[\ell_Y(t (aU+v) + (1-t) (a'U+v'))
-t\ell_Y(aU+v) - (1-t)\ell_Y(a'U+v')
\Bigr] = 0. 
\end{align*}
By the convexity of $\ell_Y$, 
$$
\ell_Y(t (aU+v) + (1-t) (a'U+v'))
=t\ell_Y(aU+v) + (1-t)\ell_Y(a'U+v') 
$$
for all $t\in [0,1]$. By the strict convexity of $\ell_Y$, we must have $aU+v=a'U+v'$.  
If $a=a'$ holds then we have $v=v'$, completing the proof of uniqueness.
We now show that $a=a'$. The condition $aU+v=a'U+v'$ gives 
$$
v_t = t v + (1-t)v' = t \{v+ U(a-a')\} + (1-t)v'= v' + t U(a-a'),
$$ 
so that $\E[v_t G]=\E[v'G]$ by independence of $G$ and $U$.
Meanwhile, since $(a_t, v_t)$ solves the constrained optimization problem for all $t\in[0,1]$, the constraint is satisfied in equality $\|v_t\|-\E[v_t G]/\sqrt{1-\delta^{-1}} = 0$ for all $t\in [0,1]$ (see \Cref{lagrange}). Then, $t\mapsto \|v_t\|$ must be constant as well.
The polynomial $\|v_t\|^2$ is given
\begin{align*}
    \|v_t\|^2 = \|v' + t U(a-a')\|^2 = \|v'\|^2 + t \E[v' U(a-a')] + t^2 (a-a')^2.
\end{align*}
Since it is constant in $t$,
the quadratic term must be 0, hence $(a-a')^2=0$.

Finally, let us show the uniqueness of Lagrange multiplier. Suppose that there exists two distinct Lagrange multipliers $\mu_* \ne \mu_{**}\in\R_{>0}$ associated with the minimizer $(a_*, v_*)$. Since the subderivative $\partial_v \ml(a_*, v_*)$ is the singleton $\{\ell_Y'(a_*U+v_*)\}$ at the minimizer $(a_*, v_*)$, the KKT conditions 
$$-\mu_*\nabla\mg(v_*), -\mu_{**}\nabla\mg(v_*)\in \partial_v\ml(a_*, v_*)$$ 
lead to 
$$
-\mu_* \nabla\mg(v_*) = -\mu_{**} \nabla\mg(v_{*}) = \ell_Y'(a_*U+v_*). 
$$
Combined with $\mu_{*}\ne \mu_{**}$, this gives 
$\nabla\mg(v_*)=0$ and $\ell_Y'(a_*U+v_*) = 0$. Here $0 = \nabla\mg(v_*)=v_*/\|v_*\|-G/\sqrt{1-\delta^{-1}}$ implies $v_*= \frac{\|v_*\|}{\sqrt{1-\delta^{-1}}} G$. Letting $\sigma_*=\|v_*\|/\sqrt{1-\delta^{-1}}>0$,  substituting this to $\ell_Y'(a_*U+v_*) = 0$, we get
$$
\ell_Y'(a_*U+ \sigma_* G) = 0
$$
Since $\ell_Y$ is strictly convex, this means $\PP(\Omega_{\vee})=1$, which is a contradiction with $\PP(\Omega_{\vee})<1$. This completes the proof of uniqueness of Lagrange multiplier. 

\section{Proof of the phase transition: \Cref{theorem:phase_transition}}
\label{appendix_E_proof_3.4}

\begin{lemma}\label{lm:no_minimizer}
        If $\delta \le \delta_\infty$, the optimization problem \eqref{eq:infinite_dimensional_optimization} does not admit any minimizer. 
\end{lemma}
    
\begin{proof}
        Let us take $(t_*, p_*)\in \mathcal{C}\subset \R\times \mh$ as in \Cref{lemma:threshold_exist} so that 
        $$
        \|p_*\|=\sqrt{1-\delta_\infty^{-1}}, \quad \E[p_*^2]=\E[p_*G]. 
        $$
        Substituting this to $\mg(p_*) = \|p_*\|-\E[p_*G]/\sqrt{1-\delta^{-1}}$, using the condition $\delta\le \delta_\infty$, we have 
        $$
        \mg(p_*)= \|p_*\| \Bigl(1-\frac{\sqrt{1-\delta_{\infty}^{-1}}}{\sqrt{1-\delta^{-1}}}\Bigr) \le 0,
        $$
        i.e., $p_*$ satisfies the constraint $\mg\le 0$ under the assumption $\delta \le \delta_{\infty}$. 
    
    Let us fix $(a,v)\in \operatorname{dom} \ml \cap \operatorname{lev}_{\le 0} \mg$ so that $\E[|\ell_Y(aU+v)|]<+\infty$ and $\mg(v)\le 0$. For any $\nu\ge 1$, consider the convex combination with coefficients $(1-1/\nu)$ and $1/\nu$
    given by
    \begin{equation*}
       (a_\nu, v_\nu) \coloneq (\a,v)(1-1/\nu) + (1/\nu)(\nu t_*,\nu p_*).
    \end{equation*}
    Note that $(a_\nu, v_\nu)\to (a, v)+(t_*, p_*)$ almost surely as $\nu\to+\infty$. 
    By the convexity of $\mg$, the convex combination $(a_\nu, v_\nu)$ also satisfies the constraint $\mg(v_\nu)\le 0$. On the other hand, the convexity of $\ell_Y$ implies 
    $$
    \ell_Y(a_\nu U + v_\nu)\le (1-1/\nu) \ell_Y(aU+v) + ( 1/\nu) \cdot \ell_Y(\nu t_*U + \nu p_*).
    $$
    Here, by the definition of $v_*$,
    \begin{align*}
        \ell_Y(\nu t_*U + \nu p_*) = 
        \begin{cases}
            \ell_Y(\nu (G+Ut_*)_-) & \text{under $\Omega_{\nearrow}$}\\
            \ell_Y(\nu (G+Ut_*)_+) & \text{under $\Omega_{\searrow}$}\\
            \ell_Y(0) & \text{under $\Omega_{\vee}$}
        \end{cases}
    \end{align*}
    and hence $\ell_Y(\nu t_*U + \nu t_*) $ is uniformly upper bounded by $\ell_Y(0)$ for all $\nu\ge 1$. We have 
    \begin{align}\label{eq:upper_estimate_nu}
        \ell_Y(a_\nu U + v_\nu)\le (1-1/\nu) \ell_Y(aU+v) + ( 1/\nu) \ell_Y(0).     
    \end{align}
    Taking expectation, we get 
    \[
    \ml(a_\nu, v_\nu)
        \le
        (1-1/\nu) \ml(\a, v)
        +
        (1/\nu) \E[\ell_Y(0)]. 
    \]
    Now we consider the limit $\nu\to+\infty$. The RHS converges to $\ml(\a,v)$ as $\nu\to+\infty$ since we took $(a, \nu)\in \operatorname{dom}\ml$ and $\E[|\ell_Y(0)|] < +\infty$ by \Cref{lm:ell_0_bounded}. 
    As for the LHS $\ml(a_\nu, v_\nu)=\E[\ell_Y(a_\nu U + v_\nu)]$, from \eqref{eq:upper_estimate_nu}, the integrand is uniformly bounded as 
    $$
    |\ell_Y(a_\nu U + v_\nu)| \le |\inf_u\ell_Y(u)| + |\ell_Y(aU+v)| + |\ell_Y(0)|
    $$
    where the RHS has a finite expectation. Therefore, by the dominated convergence theorem, we get 
    \begin{align*}
        \lim_{\nu\to +\infty }\ml(a_\nu, v_\nu) &= \lim_{\nu\to +\infty } \E[\ell_Y(a_\nu U + v_\nu)]\\
        &= \E[ \lim_{\nu\to +\infty }  \ell_Y(a_\nu U + v_\nu)]\\
        &= \E[\ell_Y((a+t_*)U + v+p_*)] && \text{by $(a_\nu, v_\nu)\to (a+t_*,v+p_*)$ and the continuity of $\ell_Y(\cdot)$}\\
        &= \ml(a+t_*, v+p_*).
    \end{align*}
    We have proved that for any $(\a,v)\in \operatorname{dom}\ml$, the 
    inequality
    $\ml(\a+t_*,v + p_*) \le \ml(\a,v)$ holds.
    
    Now we suppose that a minimizer $(a_{\min}, v_{\min})$ exists. Then it must be unique by \Cref{lm:unique}. Applying the inequality $\ml(\a+t_*,v + p_*) \le \ml(\a,v)$ we have established with $(a, v)=(a_{\min}, v_{\min})$, we must have $t_* = p_* =0$. Substituting this to the definition of $p_*$, we are left with
    $$
    (G)_{-} I\{\Omega_{\nearrow}(Y)\} + (G)_+ I\{\Omega_{\searrow}(Y)\} = 0. 
    $$
    Taking the expectation of this, since $G$ and $Y$ are independent and $\E[(G)_+]=\E[(G)_-]=\sqrt{2/\pi}>0$, we obtain $\PP(\Omega_{\nearrow}(Y))+ \PP(\Omega_{\searrow})=0$, which contradicts the assumption
   $\PP(\Omega_{\vee})=1-\PP(\Omega_{\nearrow}(Y)) - \PP(\Omega_{\searrow})<1$. Therefore, the minimizer does not exist. 
\end{proof}

\begin{lemma}\label{lm:minimizer_exist}
Assume that there exists a positive constant $b$ and positive random variable $D(Y)$ with $\E[D(Y)]<+\infty$ such that  
\begin{align*}
    \ell_Y(u) \ge -D(Y) + \frac{1}{b} \times  
    \begin{cases}
        u  &  \text{under }\Omega_{\nearrow}\\
        -u &  \text{under }\Omega_{\searrow}\\
        |u|& \text{under }\Omega_{\vee}
\end{cases}
\end{align*}
Suppose $\delta > \delta_\infty$. Then for any deterministic $\xi>0$, if
$(a,v)$ satisfies $\mg(v)\le 0$
and $\ml(a, v) \le \xi$, 
then $\a + \|v\|\le C(\xi)$ for some constant depending on $\xi$.
Consequently, the objective function of the minimization problem
is coercive and admits a minimizer $(\a_*,v_*)$ by Proposition 11.15 in \cite{bauschke2017convex}.
\end{lemma}

\begin{proof}
By \Cref{lemma:threshold_exist}, $\delta_{\infty}$ can be represented as 
$\delta_{\infty}^{-1} = \min_{(t, p)\in \mathcal{C}} \E[(G-p)^2]$, where $\mathcal{C}\subset \R\times \mh$ is the cone defined as 
\begin{align*}
    (a, v)\in \mathcal{C} \Leftrightarrow
    (v+aU) \begin{cases}
        \le 0  &\text{under }\Omega_{\nearrow}\\
        \ge 0  &\text{under }\Omega_{\searrow}\\
        = 0  & \text{under }\Omega_{\vee}\\
    \end{cases}
\end{align*}
and the optimal $(t_*, p_*)\in \argmin_{(t, p)\in \mathcal{C}} \E[(G-p)^2]$ satisfies 
\begin{align}\label{eq:ineq_second_proof_delta}
\|p_*\| = \sqrt{1-\delta_{\infty}^{-1}} <\sqrt{1-\delta^{-1}}, \quad \E[p_*^2]=\E[p_*G]
\end{align}
where we have used $\delta>\delta_{\infty}$. 
Now for all $(a, v)\in\R\times \mh$, define $\tilde{v}\in\mh$ as 
\begin{align*}
   \tilde{v} =-aU + \begin{cases}
       (aU+v)_{-} & \text{under }\Omega_{\nearrow}\\
        (aU+v)_{+} & \text{under }\Omega_{\searrow}\\
        0 & \text{under }\Omega_{\vee}
    \end{cases}
\end{align*}
so that $(a, \tilde{v})\in \mathcal{C}$ for all $(a, v)\in\R\times \mh$. Furthermore, 
\begin{align*}
    v-\tilde{v} = \begin{cases}
        (v+aU)_+ & \text{under }\Omega_{\nearrow}\\
        (v+aU)_- & \text{under }\Omega_{\searrow}\\
        v+aU & \text{under }\Omega_{\vee}
    \end{cases}
    \quad \text{and consequently }
    \quad \|v-\tilde{v}\| \le \|v\| + |a|
\end{align*}
By the condition $\ml(a, v)=\E[\ell_Y(aU+v)]\le \xi$ and the using the variable $D(Y)$ in \Cref{assumption}, we know
\begin{align*}
    |v-\tilde{v}| &\le D(Y) + b \begin{cases}
        \ell_Y((v+aU)_+) & \text{under }\Omega_{\nearrow}\\
        \ell_Y((v+aU)_-) & \text{under }\Omega_{\searrow}\\
        \ell_Y(v+aU)& \text{under }\Omega_{\vee} 
    \end{cases}\\
    &= D(Y) + b \ell_Y(aU+v) +  b \begin{cases}
        \ell_Y(0) - \ell_Y(v+aU) & \text{if }\Omega_{\nearrow} \text{ and } aU+v \le 0\\
        \ell_Y(0) - \ell_Y(v+aU) & \text{if }\Omega_{\searrow} \text{ and } aU+v \ge 0\\
      0 & \text{otherwise}
    \end{cases}\\
    &\le D(Y) + b \ell_Y(aU+v) +  b|\ell_Y(0)-\inf_u \ell_Y(u)| 
\end{align*}
so that
$$
\E[|v-\tilde{v}|] \le b\xi + \E[D(Y)] + b (\E[|\ell_Y(0)|] + \E[|\inf_u\ell_Y(u)|]) = C^{(1)}(\xi). 
$$
By \Cref{assum:so_not_trivial}, we can take a sufficiently small $u_0>0$ such that 
\begin{align*}
    \Omega_{+} &\coloneq \Bigl\{\{U>u_0\} \cap \Omega_{\nearrow}\Bigr\} \cup \Bigl\{\{U<-u_0\} \cap \Omega_{\searrow} \Bigr\} \cup  \Bigl\{ \{|U|>u_0\}\cap \Omega_{\vee}\Bigr\}\\
    \Omega_{-} &\coloneq \Bigl\{\{U<-u_0\} \cap \Omega_{\nearrow}\Bigr\} \cup \Bigl\{\{U>u_0\} \cap \Omega_{\searrow} \Bigr\} \cup  \Bigl\{ \{|U|>u_0\}\cap \Omega_{\vee}\Bigr\}
\end{align*}
have positive probabilities. Let $p_0=\min(\PP(\Omega_{-}), \PP(\Omega_{+})) > 0$. 

Suppose $a>0$. Under $\{U>u_0\} \cap \Omega_{\nearrow}$, it holds that $aU+v \le b \ell_Y(aU+v) + D(Y)$ and hence 
$$
a u_0 \le aU \le b \ell_Y(aU+v) + D(Y) - v.
$$
Under $\{U<-u_0\} \cap \Omega_{\searrow}$, we have  $-(aU+v) \le b \ell_Y(aU+v) + D(Y)$ and hence
$$
a u_0 \le a(-U) \le  b \ell_Y(aU+v) + D(Y) + v
$$
Finally, under $\{|U|>u_0\}\cap\Omega_{\vee}$, we have $|a U + v| \le b \ell_Y(aU+v) + D(Y) + v$ so that 
$$
a u_0 \le a|U| \le  |aU+v|+|v|\le  b \ell_Y(aU+v) + D(Y) + |v|.
$$
Combining them all together,
\begin{align*}
    au_0 \PP(\Omega_{+}) & \le  \E[I\{\Omega_+\} a u_0] \\
    &\le b \E[(\ell_Y(aU+v) + D(Y) + |v|) I\{\Omega_{+}\}]\\
    &= b \E[(\underbrace{\ell_Y(aU+v) - \inf_{u}\ell_Y(u)}_{\ge 0} + \inf_u\ell_Y(u) + \underbrace{D(Y) + |v|}_{\ge 0}) I\{\Omega_{+}\}]\\
    &\le b \E[\ell_Y(aU+v) - \inf_{u}\ell_Y(u)] + \E[\inf_u\ell_Y(u) I\{\Omega_+\}] + \E[D(Y)] + \E[|v|]\\
    &= b \E[\ell_Y(aU+v)] - \E[\inf_u\ell_Y(u) I\{\Omega_+^c\}] + \E[D(Y)] + \E[|v|]\\
    &\le b \xi  + \E[|\inf_u\ell_Y(u)|] + \E[D(Y)] + \|v\|_2.
\end{align*}
By the same argument, if $a<0$, considering the event $\Omega_{-}$, we have 
$$
(-a) u_0 \PP(\Omega_{-}) \le  b\xi + \E[|\inf_u\ell_Y(u)|] + \E[D(Y)] + \|v\|_2
.
$$
Combined with $\min(\PP(\Omega_{+}), \PP(\Omega_{-})) =p_0>0$, we have that
$$
|a| \le (u_0 p_0)^{-1} \Bigl(\xi + \E[|\inf_u\ell_Y(u)|] + \E[D] + \|v\|_2\Bigr) \le \C(\xi)(1+\|v\|_2)
$$
With $\|v-\tilde{v}\|\le |a|+\|v\|$, we obtain
\begin{align}\label{eq:bound-L2-tildev-v}
    \|v-\tilde{v}\| \le C^{(2)}(\xi) (1 + \|v\|_2).     
\end{align}

Now we claim that the Pythagorean inequality
\begin{equation}
    \label{pythagorean}
    \E[(G-p_*)^2] + \E[(p_*-\tilde v)^2] \le \E[(G-\tilde v)^2] 
\end{equation}
holds.
A proof of \eqref{pythagorean} is as follows.
Define for $s\in[0,1]$ the convex combination
$v_s = p_* (1-s) + s \tilde v$ and consider the function
$$
\varphi(s) = \E[(G-v_s)^2] - \E[(G-p_*)^2]- \E[(v_s-p_*)^2].
$$
By the definition of $v_s$, it follows that $\varphi(0)=0$. Furthermore, $\varphi(s)$ is linear in $s$, as the quadratic term $\E[v_s^2]$ cancels out. On the other hand, the optimality of $(t_*, p_*)\in \argmin_{(t,p)\in\mathcal{C}}\E[(G-p)^2]$ implies that $\E[(G-p_*)^2] \le \E[(G-v_s)^2]$. Combining this with the definition of $\varphi(s)$, we have 
$\varphi(s) \ge - \E[(v_s - p_*)^2] = - s^2 \E[(p_*-\tilde{v})^2].$
Since $\varphi(s)$ is linear in $s$ and satisfies $\varphi(s)\ge -O(s^2)$ as $s\to0$,
the slope $\varphi'(0)$ must be non-negative. This implies $0 = \varphi(0)\le \varphi(1)$, which is exactly inequality \eqref{pythagorean}.

In both sides of \eqref{pythagorean},
$\E[\tilde v^2]$ cancel out, $\E[G^2]=1$ and
$\E[(G-p_*)^2]=1-\|p_*\|^2$ by \eqref{eq:ineq_second_proof_delta}, so that \eqref{pythagorean}
can be rewritten as
\begin{equation}
    -2 \E[p_*\tilde v]\le -2\E[\tilde v G ],
    \qquad
    \text{i.e., } \qquad
    \E[\tilde v(p_*-G)]\ge 0
    .
    \label{25_tilde_v_p_star}
\end{equation}
From this, for all $v$ satisfying the constraint $\mg(v) =\|v\| - \E[vG]/\sqrt{1-\delta^{-1}} \le 0$, 
\begin{align}
    \Bigl(\sqrt{1-\delta^{-1}}-\|p_*\|\Bigr)\|v\| & \le \E[vG] - \|p_*\| \|v\| && \text{by $\mg(v) =\|v\| - \E[vG]/\sqrt{1-\delta^{-1}} \le 0$} \nonumber \\
    &\le \E[v(G-p_*)] && \text{by the Cauchy-Schwarz inequality $\E[p_*v]\le \|p_*\|\|v\|$}\nonumber \\
    &\le \E[(\tilde v - v)(p_*-G)] && \text{by $\E[\tilde v(p_*-G)]\ge 0$ in \eqref{25_tilde_v_p_star}. } \label{25_serependipitous}
\end{align}
The parenthesis $\sqrt{1-\delta^{-1}}-\|p_*\|$ on the left is positive thanks to \eqref{eq:ineq_second_proof_delta}. On the right-hand side, considering the event $I\{|p_*-G| > C\}$ for some constant $C>0$ to be specified later, the Cauchy-Schwarz inequality gives
\begin{align*}
    \E[(\tilde{v}-v)(p_*-G)] &\le C \E[|\tilde{v}-v|] + \E[I\{|p_*-G|> C\}|p_*-G| |\tilde{v}-v|] \\
    &\le C \E[|\tilde{v}-v|] + \sqrt{\E[I\{p_*-G|\ge C\}|p_*-G|^2]} \cdot \|\tilde{v}-v\|. 
\end{align*}
Since $p_*-G$ is bounded in L2, we can take a sufficiently large $C^{(3)}(\epsilon)$ for all $\epsilon>0$  such that $\E[I\{|p_*-G| > C^{(3)}(\epsilon)\} |p_*-G|^2] \le \epsilon^2$. Thus, we have that for all $\epsilon>0$, 
\begin{align}
    \Bigl(\sqrt{1-\delta^{-1}}-\|p_*\|\Bigr)\|v\|
    &\le C^{(3)}(\epsilon)\E|\tilde v -v| + \epsilon \|v-\tilde v\|_2.
\end{align}
Recall the bound $\|v-\tilde{v}\|\le C^{(2)}(\xi) (1 + \|v\|)$ in \eqref{eq:bound-L2-tildev-v}, so if we take a sufficiently small $\epsilon=\epsilon_0$ such that $ \epsilon_0 C^{(2)}(\xi)  \le (\sqrt{1-\delta^{-1}}-\|p_*\|)/2$, we obtain
\begin{equation}
    \frac 1 2 \bigl(\sqrt{1-\delta^{-1}}-\|p_*\|\bigr)\|v\|
    \le C(\epsilon_0) \E|\tilde v - v| + \epsilon_0 C^{(2)}(\xi)
\end{equation}Since 
$\E|\tilde v - v|\le C^{(1)}(\xi)$ was established at the beginning,
this completes the proof. 
\end{proof}

\end{document}